\documentclass[12pt,a4paper,reqno]{amsart}

\usepackage[sc, noBBpl]{mathpazo}	
\usepackage[bbgreekl]{mathbbol}
\usepackage{amssymb, amsmath, amsthm, amsfonts}

\DeclareSymbolFontAlphabet{\mathbb}{AMSb}
\DeclareSymbolFontAlphabet{\mathbbl}{bbold}

\usepackage[T1]{fontenc}
\usepackage[utf8]{inputenc}
\usepackage{graphicx}
\usepackage[linkbordercolor={0.8 0.9 0.9}]{hyperref}
\usepackage{comment}
\usepackage{latexsym,epsfig}
\usepackage[all]{xy}
\usepackage[usenames,dvipsnames]{color}

\usepackage[normalem]{ulem}

\newcommand{\st}{\;|\;}
\newcommand{\half}{\frac{1}{2}}

\DeclareMathOperator{\Aut}{Aut}

\DeclareMathOperator{\Exp}{Exp}
\DeclareMathOperator{\Op}{Op}
\DeclareMathOperator{\ord}{ord}

\DeclareMathOperator{\supp}{supp}

\DeclareMathOperator{\WF}{WF}

\newcommand{\lie}{\mathfrak}

\renewcommand{\AA}{\mathbb{A}}

\newcommand{\DD}{\mathbb{D}}

\newcommand{\KK}{\mathbb{K}}
\newcommand{\UU}{\mathbb{U}}
\newcommand{\VV}{\mathbb{V}}
\newcommand{\NN}{\mathbb{N}}
\newcommand{\RR}{\mathbb{R}}

\newcommand{\TT}{\mathbb{T}}
\newcommand{\PP}{\mathbb{P}}
\newcommand{\QQ}{\mathbb{Q}}
\newcommand{\ZZ}{{\mathbb{Z}}}
\newcommand{\II}{\mathbb{I}}

\newcommand{\PPsi}{\mathbbl{\Psi}}

\newcommand{\sD}{\mathcal{D}}

\newcommand{\sE}{\mathcal{E}}
\newcommand{\sDp}{\mathcal{D}_\mathrm{p}}
\newcommand{\sS}{\mathcal{S}}
\newcommand{\sF}{\mathcal{F}}
\newcommand{\sU}{\mathcal{U}}

\newcommand{\bfa}{\mathbf{a}}
\newcommand{\bfb}{\mathbf{b}}
\newcommand{\bfc}{\mathbf{c}}

\newcommand{\DO}{\mathrm{DO}}
\newcommand{\Cp}{C_\mathrm{p}}
\newcommand{\Cc}{C_\mathrm{c}}
\newcommand{\Dens}{\Cp^\infty(\TTHM;\Omega_r)}

\newcommand{\TTM}{\mathbb{T}M}
\newcommand{\tHM}[1][]{\lie{t}^{#1}_HM}
\newcommand{\tsHM}{\lie{t}^*_HM}
\newcommand{\THM}{T_HM}

\newcommand{\TTHM}{\mathbb{T}_HM}
\newcommand{\aHG}[1][]{\lie{a}_H^{#1}G}
\newcommand{\AG}{\mathrm{A}G}
\newcommand{\AHG}{\mathrm{A}_H G}
\newcommand{\AAHG}{\mathbb{A}_HG}
\newcommand{\aaHG}{\mathbbl{a}_HG}

\newcommand{\EExp}{\mathbb{E}\mathrm{xp}}

\newcommand{\id}{\mathrm{id}}
\newcommand{\zoom}{\alpha}
\newcommand{\slot}{\,\cdot\,}

\newcommand{\into}{\hookrightarrow}

\newcommand{\PPtilde}{\widetilde{\PP}}

\newcommand{\ang}[1]{\langle #1 \rangle} 
\newcommand{\ip}[1]{\langle #1 \rangle} 

\newtheorem{theorem}{Theorem}
\newtheorem{lemma}[theorem]{Lemma}
\newtheorem{proposition}[theorem]{Proposition}
\newtheorem{corollary}[theorem]{Corollary}

\theoremstyle{definition}
\newtheorem{definition}[theorem]{Definition}

\newtheorem{example}[theorem]{Example}

\theoremstyle{remark} 
\newtheorem{remark}[theorem]{Remark}

\title{A groupoid approach to pseudodifferential calculi}
\author{Erik van Erp}
\address{Department of Mathematics, Dartmouth College, Hanover, NH 03755, USA}
\email{erikvanerp@dartmouth.edu}

\author{Robert Yuncken}
\thanks{R.~Yuncken was supported by the project SINGSTAR of the Agence Nationale de la Recherche, ANR-14-CE25-0012-01.}
\address{Université Clermont Auvergne, Université Blaise Pascal, BP 10448, F-63000 Clermont-Ferrand, France}
\email{yuncken@math.univ-bpclermont.fr}

\date{}         

\begin{document}

\begin{abstract}
We give an algebraic/geometric characterization of the classical pseudodifferential operators on a smooth manifold in terms of the tangent groupoid and its natural $\mathbb{R}^\times_+$-action.  Specifically, a properly supported semiregular distribution on $M\times M$ is the Schwartz kernel of a classical pseudodifferential operator if and only if it extends to a smooth family of distributions on the range fibres of the tangent groupoid that  is homogeneous for the $\mathbb{R}^\times_+$-action modulo smooth functions.  Moreover, we show that the basic properties of pseudodifferential operators can be proven directly from this characterization.  Further, with the appropriate generalization of the tangent bundle, the same definition applies without change to define pseudodifferential calculi on arbitrary filtered manifolds, in particular the Heisenberg calculus.
\end{abstract}

\maketitle


\section{Introduction}


In this article we propose a simple and coordinate-free geometric\slash algebraic definition of pseudodifferential operators that  recovers the classical pseudodifferential operators, but that  also applies without change to construct the Heisenberg calculus ({\em cf}.\ \cite{BeaGre, Taylor:microlocal}) and its generalizations to other filtered manifolds (\emph{cf}.\cite{Cummins, ChrGelGloPol, Melin:Lie_filtrations, Ponge:Heisenberg}).

Our  method is built upon the tangent groupoid.
 The tangent groupoid was used by Connes \cite{Connes:NCG} as a geometric device for glueing a pseudodifferential operator to its principal symbol.  We show that classical pseudodifferential operators can be completely characterized using the tangent groupoid.
In this approach, the {\em smoothness} of the tangent groupoid  accounts for the full asymptotic expansion of the symbol of a pseudodifferential operator, not just its principal term.

Furthermore, we will show that one can prove the basic properties of pseudodifferential operators directly from this geometric characterization.  This allows us to define more general pseudodifferential calculi without needing to prescribe symbol or kernel classes explicitly.

\subsection{Outline of the strategy}
\label{sec:outline}

Let $D\colon C^\infty(M)\to C^\infty(M)$ be a differential operator of order $m$, given in local coordinates by
\[ D=\sum_{|\alpha|\le m} a_\alpha\partial^\alpha.\]
Its highest order part at each point $x\in M$,
\[ D_x = \sum_{\alpha=m} a_\alpha(x)\partial^\alpha,\]
can be viewed as a smooth family of constant coefficient operators $D_x$ on the tangent spaces $T_xM$.
The family $\{D_x\}_{x\in M}$ is the principal `cosymbol' of $D$, i.e.\ the inverse Fourier transform of the principal symbol of $D$.
The operator $D$ and its cosymbol $\{D_x\}_{x\in M}$ can be smoothly assembled into a single differential operator $\DD$ on the tangent groupoid $\TTM$.

Algebraically, $\TTM$ is the disjoint union of the abelian groups $T_xM\cong \RR^n$ and a one parameter family of pair groupoids $M\times M$:
\[ \TTM = (TM\times\{0\}) \; \sqcup \; (M\times M\times \RR^\times).\]
The non-trivial aspect of the tangent groupoid is its smooth structure.

Let $\DD$ be the differential operator on $\TTM$ that restricts to $D_x$ on $T_xM$ and to $t^mD$ on each $M\times M\times \{t\}$ with $t\ne 0$, acting on the second factor in $M\times M$.
The smooth structure on $\TTM$ is precisely such that the coefficients of  $\DD$ are $C^\infty$ functions on $\TTM$.
Thus $\DD$ is  a left-invariant differential operator on $\TTM$ that glues the operator $D$ (at $t=1$) to its principal cosymbol (at $t=0$).

Let us reformulate this in terms of distributional kernels on the groupoid $\TTM$ (for full details see Example \ref{ex:diff_op}).
Since each $D_x$ is a translation invariant operator on the abelian group $T_xM$,
there exists a distribution $u_x$ on $T_xM$ such that $D_x$ acts as convolution with $u_x$,
\[D_x\varphi = u_x\ast \varphi\qquad u_x \in \sE'(T_xM).\]
Likewise,  $D$  has a Schwartz kernel $k$ that  is smooth in the first variable, meaning that the function
\[ k: x\mapsto k_x=D^t\delta_x\in \sE'(M)\]
is a smooth map from $M$ to $\sE'(M)$.
If we now define
\begin{align}
 \PP_{(x,0)}&=u_x\in \sE'(T_xM), & t=0 \nonumber \\
 \PP_{(x,t)}&=t^mk_x\in \sE'(M), & t\ne 0 \label{eq:diff_op_intro}
\end{align}
then $(\PP_{(x,t)})_{(x,t)\in M\times \RR}$
is a smooth family of compactly supported distributions in the fibres of the range map $r\,\colon \TTM\to \TTM^{(0)}= M \times \RR$.

The space of such smooth families of distributions on the $r$-fibres is denoted $\sE'_r(\TTM)$; see Section \ref{section:conv} or \cite{LesManVas}.  In analogy with the situation for Lie groups, these distributions form a convolution algebra in which the smooth functions---or more accurately, the smooth densities along the $r$-fibres---sit as a right ideal.

We need one more piece of structure.
The tangent groupoid has a one parameter family of Lie groupoid automorphisms
\begin{equation}
  \label{eq:classical_zoom}
  \alpha\;\colon\;\RR^\times_+\mapsto\mathrm{Aut}(\TTM)
\end{equation}
defined by
\begin{align*}
  \zoom_\lambda (x,y,t) &= (x,y,\lambda^{-1}t) &&(x,y,t)\in M\times M\times \RR^\times,\\
  \zoom_\lambda (x,\xi,0) &= (x, \delta_\lambda\xi,0) &&\xi\in T_xM, 
\end{align*}
where $\delta_\lambda$ denotes dilation of tangent vectors, $\delta_\lambda\xi = \lambda\xi$.
Note that the smooth family of distributions $\PP$ is homogeneous of weight $m$ for the action $\alpha$.
In fact, the differential operators are characterized by this property.

\begin{proposition}
 \label{prop:DOs}
A semiregular kernel $P \in \sE_r'(M\times M)$ is the Schwartz kernel of a differential operator of order $m$ if and only if $P = \PP|_{t=1}$ for some $\PP \in \sE_r'(\TTM)$ such that $\alpha_{\lambda*}\PP = \lambda^m\PP$ for all $\lambda\in \RR^\times_+$.
\end{proposition}

The motivating observation of this paper is that pseudodifferential operators admit a similar characterization.

\begin{theorem}
 \label{thm:PsiDOs}
A semiregular kernel $P \in \sE_r'(M\times M)$ is the Schwartz kernel of a properly supported classical pseudodifferential operator of order $m$ if and only if $P = \PP|_{t=1}$ for some $\PP \in \sE_r'(\TTM)$ such that $\alpha_{\lambda*}\PP - \lambda^m\PP$ is a smooth density for all $\lambda\in \RR^\times_+$.
\end{theorem}

Now, if we replace $\TTM$ by the appropriate tangent groupoid $\TTHM$
for a contact manifold \cite{VanErp:Thesis, Ponge:Groupoid}, or a more general filtered manifold \cite{ChoPon, VanYun:groupoid}, then the {\em same} definition can be used to produce versions of the Heisenberg calculus \cite{BeaGre, Taylor:microlocal} or Melin's calculus \cite{Melin:Lie_filtrations}.
In this way we show that, once an appropriate tangent groupoid has been constructed, the details of the corresponding pseudodifferential calculus follow automatically.

In particular, we prove, in the general setting of
filtered manifolds, 
that our pseudodifferential operators form a $\ZZ$-filtered algebra $\Psi_H^\bullet(M)$; 
that restriction to $t=0$ in $\sE_r'(\TT_HM)$ assigns a principal cosymbol to an operator in the calculus; 
that there is a short exact sequence
\begin{equation}
 \label{eq:short_exact_sequence}
 0\to \Psi_H^{m-1}(M)\to \Psi_H^m(M) \stackrel{\sigma_m}{\longrightarrow} \Sigma^m_H(M)\to 0 , 
\end{equation}
where $\Sigma^m_H(M)$ is the (noncommutative) algebra of principal cosymbols;
that the intersection $\cap_{m\in \ZZ} \Psi^m(M)$ is the ideal of properly supported smoothing operators $\Psi^{-\infty}(M)$;
that $\Psi_H^\bullet(M)$ is complete with respect to asymptotic expansions;
and that if an operator in $\Psi_H^m(M)$ has invertible  principal cosymbol,
then it has a parametrix in $\Psi_H^{-m}(M)$.

As an application, we immediately obtain the hypoellipticity of `elliptic' elements in our calculi, where ellipticity here is taken in the abstract sense of having an invertible principal cosymbol%
\footnote{
Since the appearance of this work as a preprint, Dave and Haller, building on \cite{ChrGelGloPol} and \cite{Ponge:Heisenberg}, have shown that our abstract ellipticity condition is equivalent to the pointwise Rockland condition.  See \cite{DavHal} for details.
}. 
Moreover, given that there is a tangent groupoid associated to these calculi, Connes tangent groupoid method \cite{Connes:NCG} should yield abstract index theorems for such operators, as is the case for subelliptic operators on a contact manifold \cite{VanErp:I}.

\begin{remark}
A different geometric approach to the Heisenberg calculus is that of Epstein and Melrose \cite{EpsMel:book} using compactifications of the cotangent bundle.  This approach is currently limited to the standard calculus and the Heisenberg calculus, but it has other notable advantages, for instance an explicit noncommutative product on the symbol algebra and even a fusion of the standard and Heisenberg calculi.  For comments on the relationship between the two, see Section \ref{sec:symbols}.
\end{remark}

\subsection{Structure and history of the paper}

In Section \ref{section:conv} we briefly recap convolution algebras of distributions on Lie groupoids.  Section \ref{sec:tangent_groupoid} recaps the tangent groupoid of a filtered manifold.  This tangent groupoid construction, which is fundamental for the present work, was originally presented in an earlier version of this paper (arXiv preprint \href{https://arxiv.org/abs/1511.01041v3}{1511.01041 Version 3}).  Following the suggestion of a referee, that construction now appears in a separate article, \cite{VanYun:groupoid}.  See also \cite{ChoPon, HigSad} for other points of view on this tangent groupoid.

The main new results of this paper are contained in Sections \ref{sec:PsiDOs}--\ref{sec:H-ellipticity}, where we define our pseudodifferential operators and prove their fundamental properties.  In Section \ref{sec:diff_ops} we examine the particular case of differential operators.  In Section \ref{sec:classical_PsiDOs} we prove that our construction recovers the classical pseudodifferential operators in the case of the classical tangent groupoid.

Finally, in Section \ref{sec:filtered_groupoid}, we show that our construction can be generalized to yield right-invariant pseudodifferential operators on the leaves of an arbitrary Lie groupoid, in the sense of \cite{NisWeiXu, MonPie}, perhaps with an additional filtered structure on the Lie algebroid.  Thanks to existing work by many authors, \emph{e.g.}\ \cite{Connes:integration, NisWeiXu, AmmLauNis, DebLesRoc}, this yields well-known classes of pseudodifferential operators associated to singular structures such as foliations, manifolds with boundary, and stratified manifolds.

\subsection{Acknowledgements}

This article was inspired by an observation of Debord and Skandalis in their paper \cite{DebSka} that   provides the first abstract characterization of the classical pseudodifferential operators in terms of the $\RR^\times_+$-action on the tangent groupoid.  We wish to thank them for many discussions,  particularly during the time that the first author was {\em Professor Invit\'e} at the Universit\'e Blaise Pascal, Clermont-Ferrand II.  Sincere thanks also go to Jean-Marie Lescure and Nigel Higson.



\section{Convolution of distributions on Lie groupoids}\label{section:conv}

The relevant framework for our approach to pseudodifferential operators is  convolution  of  distributions on a Lie groupoid (see \cite{LesManVas}). We  summarize the main points.

\subsection{Lie groupoids}
\label{sec:notation}

For the basic theory of Lie groupoids, see \cite{Mackenzie} or \cite{MoeMrc}.

Algebraically, a groupoid $G$ is a small category with inverses.
The set of objects of the groupoid is denoted $G^{(0)}$,
while the set of morphisms is $G^{(1)}$.
Each morphism $x\in G^{(1)}$ has a source object $s(x)\in G^{(0)}$ and range object $r(x)\in G^{(0)}$.
It is customary to let $G=G^{(1)}$, and to identify each object $u\in G^{(0)}$
with the unique identity morphism of $u$, so that $r(u)=s(u)=u$.
Thus, $G^{(0)}$ is treated as a subset of $G=G^{(1)}$.
We refer to morphisms simply as {\em elements} of the groupoid $G$,
while objects are called {\em units}.

The set of composable morphisms is
\[ G^{(2)} := \{(x,y)\in G\times G\mid s(x)=r(y)\}\]
 following the convention that morphisms point from right to left.
Composition of morphisms gives a map $G^{(2)}\to G$, $(x,y)\mapsto xy$,
conceived as multiplication of groupoid elements.

A  {\em Lie groupoid} is a groupoid $G$ that is also a  smooth manifold,
such that the space of units $G^{(0)}$ is a smooth submanifold of $G$;
source and range maps $s,r : G \to G^{(0)}$  are smooth submersions;
multiplication is a smooth map $G^{(2)}\to G$;
and the inverse map $x\mapsto x^{-1}$ is smooth $G\to G$.

Note that $G^{(2)}$  is a submanifold of $G\times G$ because $r,s$ are submersions. 
We write $G^x = r^{-1}(\{x\})$ and $G_x=s^{-1}(\{x\})$ for the range and source fibres at $x\in G^{(0)}$,
which are submanifolds of $G$.

\begin{example}
\label{ex:pair_groupoid}
The \emph{pair groupoid} $M\times M$ of a smooth manifold $M$ has  range and source maps $r(x,y)=x$, $s(x,y) = y$ and composition $(x,y)(y,z) = (x,z)$.  
\end{example}

\begin{example}
\label{ex:bundle_of_groups}
 A {\em smooth bundle of Lie groups} is a groupoid  with $r=s$. Then the fibres $G_x = G^x$ are Lie groups.
In particular, the tangent bundle $TM$ is a groupoid when viewed as a smooth bundle of abelian groups.  
\end{example}

\subsection{Convolution of fibred distributions}
\label{sec:convolution}

\begin{definition}\label{Def:fibred_dist}
An \emph{$r$-fibred distribution}
\footnote{
Our $r$-fibred distributions correspond to what Lescure, Manchon and Vassout call \emph{distributions transversal to the range maps}.  
We work with densities tangent to the $r$-fibres, while they work with half-densities.
Our choice has been made so that Schwartz kernels act on smooth functions on $M$ rather than on half-densities.

Technically, this definition gives the \emph{$r$-properly supported} $r$-fibred distributions, in the sense that $r:\supp(u) \to G^{(0)}$ is proper. 
This is the appropriate support condition for the convolution algebra of $r$-fibred distributions and we will not consider any others.
}
on a Lie groupoid $G$ is a continuous $C^\infty(G^{(0)})$-linear map
\[
  u: C^\infty(G) \to C^\infty(G^{(0)}),
\]
where the $C^\infty(G^{(0)})$-module structure on $C^\infty(G)$ is induced by the pullback of functions  via $r$.
\end{definition}
We write $\ip{u,f}$ for the image of $f\in C^\infty(G)$ under $u$.
The set of $r$-fibred distributions on $G$ is denoted $\sE'_r(G)$.
Similarly, $\sE'_s(G)$ denotes the set of $s$-fibred distributions on $G$.

An $r$-fibred distribution $u$ determines a smooth family $(u_x)_{x\in G^{(0)}}$ of distributions on the $r$-fibres $u_x\in\sE'(G^x)$, 
\[
  \ip{u,f}(x) = \ip{u_x,f|_{G^x}}, \qquad f\in C^\infty(G).
\]
Likewise, $(v_x)_{x\in M}$ denotes the smooth family of distributions on the $s$-fibres associated to $v\in\sE'_s(G)$.

The \emph{convolution} $u\ast v$ of two $r$-fibred distributions $u,v\in \sE'_r(G)$ is the $r$-fibred distribution $u\ast v$ defined by
\[ 
 \ang{ u\ast v, \varphi }(x)  = \ang{ u^x(\gamma), \,\ang{ v^{s(\gamma)}(\beta), \varphi(\gamma\beta)}}
\]
for  a test function $\varphi \in C^\infty(G)$ and  $x\in G^{(0)}$. Here $\gamma\in G^x$ and $\beta\in G^{s(\gamma)}$ are place-holder variables.
Likewise, convolution of $s$-fibred distributions $u, v\in \sE'_s(G)$ is given by
\[ 
 \ang{u\ast v, \varphi}(x) = \ang{v_x(\gamma), \ang{u_{r(\gamma)}(\beta), \varphi(\beta\gamma)}}
\]

\begin{example}
\label{ex:semiregular_Schwartz_kernels}
  Let $G=M \times M$ be the pair groupoid of a manifold $M$.  Then
  \[
   \sE'_r(M\times M) \cong C^\infty(M;\sE'(M)) \cong C^\infty(M) \hat\otimes \sE'(M),
  \]
  i.e., distributions that  are semiregular in the first variable. These are precisely the Schwartz kernels of continuous linear operators $C^\infty(M) \to C^\infty(M)$; see \cite[Ch.51]{Treves:PsiDOs}.
  Convolution in $\sE'_r(M\times M)$ corresponds to the usual composition law for Schwartz kernels:
  \[
   u*v(x,z) = \int_{y\in M} u(x,y)v(y,z).
  \]
\end{example}

\subsection{The ideal of smooth  densities}
\label{sec:densities}

\begin{definition}
\label{def:proper_subset}
A subset $X$ of a groupoid $G$ is \emph{proper} if both $r:X\to M$ and $s:X\to M$ are proper maps.  
\end{definition}

We denote by $\Omega_r$ and $\Omega_s$ the bundle of smooth densities tangent to the range and source fibres of $G$ respectively.  The space of smooth sections of $\Omega_r$ whose support is a proper subset of $G$ is denoted $\Cp^\infty(G;\Omega_r)$.  

\begin{proposition}[\cite{LesManVas}]
$\Cp^\infty(G,\Omega_r)$ is a right ideal in $\sE'_r(G)$ and $\Cp^\infty(G,\Omega_s)$ is a left ideal in $\sE'_s(G)$.
\end{proposition}

\begin{example}
 On the pair groupoid (Example \ref{ex:pair_groupoid}), $\Cp^\infty(M\times M; \Omega_r)$ is the algebra of Schwartz kernels of properly supported smoothing operators.  
\end{example}

\subsection{Proper distributions}
\label{sec:proper_distributions}

The  smooth densities  are not a {\em two-sided} ideal in $\sE'_r(G)$ or in $\sE'_s(G)$. 
To remedy this we take the intersection of $\sE'_r(G)$ and $\sE'_s(G)$, using a transverse measure.

A {\em transverse measure} on a Lie groupoid $G$ is a positive smooth density $\mu$ on the object space $G^{(0)}$.
Fixing a transverse measure allows us to integrate $r$- and $s$-fibred distributions to global distributions on $G$, via the maps
\begin{align*}
 \mu_r&\;\colon\; \sE'_r(G)\to \sD'(G)\; ; \qquad u\mapsto \mu\circ u=\int_{G^{(0)}} \ang{u,\slot}(x) d\mu(x), \\
 \mu_s&\;\colon\; \sE'_s(G)\to \sD'(G)\; ; \qquad u\mapsto \mu\circ u=\int_{G^{(0)}} \ang{u,\slot}(x) d\mu(x).
\end{align*}
 
\begin{definition}
\label{def:proper_distribution}
The set of {\em proper distributions} on $G$, denoted $\sDp'(G)$, is the intersection 
\[ \sDp'(G) = \mu_r(\sE'_r(G))\cap \mu_s(\sE'_s(G))\quad \subset \sD'(G).\]
An $r$-fibred distribution $u\in \sE'_r(G)$ is called {\em proper} if $\mu_r(u)$  is proper.
We denote the set of proper $r$-fibred distributions as $\sE'_{r,s}(G)$.
\end{definition}

In other words, a properly supported distribution $u\in \sD'(G)$ is proper
if it can be disintegrated along the $r$-fibres as well as along the $s$-fibres.
The sets $\sDp'(G)$ and $\sE'_{r,s}(G)$ are independent of the choice of transverse measure $\mu$.

\subsection{Proper Schwartz kernels}
\label{sec:pair_groupoid}

Proposition \ref{prop:regular} below shows that any reasonable  pseudodifferential calculus (on a manifold without boundary) is a subalgebra of the convolution algebra $\sE'_{r,s}(M\times M)$.

We fix a smooth density $\mu$ on $M$, which  allows us to identify smooth functions and smooth densities on $M$ via $f\mapsto f\mu$.  It also gives us a bilinear form on $\Cc^\infty(M)$:
\[
 \ang{f,g} = \int_M f(x)g(x) \,d\mu(x), \qquad f,g\in \Cc^\infty(M).
\]
The discussion in \cite[Ch.51]{Treves:PsiDOs} yields the following.

\begin{proposition}
\label{prop:regular}
Let $P:\Cc^\infty(M) \to \sD'(M)$ be a continuous linear operator with Schwartz kernel $p\in\sD'(M\times M)$.  The following are equivalent:
\begin{enumerate}
\item
$P$ maps $\Cc^\infty(M)$ to itself and admits a transpose $P^t: \Cc^\infty(M)\to \Cc^\infty(M)$ such that $\ang{f,Pg}=\ang{P^tf,g}$ for any $f,g\in \Cc^\infty(M)$,

\item
$P$ extends to continuous operators $\Cc^\infty(M)\to \Cc^\infty(M)$, $C^\infty(M)\to C^\infty(M)$, $\sE'(M)\to\sE'(M)$ and $\sD'(M) \to \sD'(M)$.

\item
$p$ is a proper distribution on the pair groupoid.

\item
$p$ is properly supported and semiregular in both variables.

\end{enumerate}
\end{proposition}

\section{The tangent groupoid of a filtered manifold}
\label{sec:tangent_groupoid}

The construction of the tangent  groupoid of a filtered manifold, which appeared in an earlier preprint of this article, has now been published separately in \cite{VanYun:groupoid}.  
To keep the exposition self-contained we review the key points here, and  refer to \cite{VanYun:groupoid} for details.

\subsection{The $H$-tangent groupoid}
\label{sec:H-tangent_groupoid}

\begin{definition}
\label{def:filtered_manifold}
(\cite{Melin:Lie_filtrations})
A {\em filtered manifold} is a smooth manifold $M$ equipped with a filtration of the tangent bundle $TM$ by vector bundles $M\times\{0\} = H^0 \leq H^1 \leq \cdots \leq H^N = TM$ such that $\Gamma^\infty(H^\bullet)$ is a Lie algebra filtration, i.e.
\[[ \Gamma^\infty(H^i), \Gamma^\infty(H^j) ] \subseteq \Gamma^\infty(H^{i+j}).\]
Here we are using the convention that $H^i = TM$ for $i>N$.  

Sections $X \in \Gamma^\infty(H^i)$ will be referred to as {\em vector fields of $H$-order $\leq i$}, written $\ord_H(X)\leq i$.
\end{definition}

\begin{remark}
Any manifold $M$ can be equipped with the trivial filtration of depth one: $H^1 = TM$.  
With this choice our construction recovers the classical pseudodifferential calculus.
\end{remark}

Given a filtration on $TM$, let $\tHM$ be the associated graded bundle
\[
 \tHM = \bigoplus_{i=1}^n \tHM[i], \qquad \text{where } \tHM[i]:=H^i/H^{i-1}.
\]
We denote the grading maps by 
\[\sigma_i:H^i \to \tHM[i]\]
This notation is chosen to coincide with the principal cosymbol maps; see Section \ref{sec:diff_ops}.  

For vector fields $X\in \Gamma^\infty(H^i)$, $Y\in \Gamma^\infty(H^j)$, and functions $f,g\in C^\infty(M)$, we have
\begin{equation}
\label{eq:Lie_bracket}
  [fX,gY] \equiv fg [X,Y] \mod \Gamma^\infty(H^{i+j-1}).
\end{equation}
This implies that the Lie bracket  of vector fields induces a pointwise bracket on the fibres of $\tHM$.  
Thus, $\tHM$ is a vector bundle over $M$ whose fibres are graded nilpotent Lie algebras.  
The osculating groupoid $T_H M$ equals $\tHM$ as a  vector bundle, and each fibre is equipped with the group law given by the Baker-Campbell-Hausdorff formula.

Algebraically, the $H$-tangent groupoid is the disjoint union
of the osculating groupoid $\THM$ with a family of pair groupoids indexed by $\RR^\times = \RR\setminus\{0\}$:
\begin{equation}
  \label{eq:H-tangent_groupoid}
  \TTHM = \THM \times \{0\} \; \sqcup \; (M\times M) \times \RR^\times.
\end{equation}
Separately, the components $\THM$ and $(M\times M)\times\RR^\times$  have obvious smooth structures.  
The main result of \cite{VanYun:groupoid} is that there is a compatible global smooth structure on $\TTHM$ that makes $\TTHM$ into a Lie groupoid.
For our construction of pseudodifferential operators
a precise description of the smooth structure of $\TTHM$ is important.
We review the details relevant for this paper in the next section, and refer to \cite{VanYun:groupoid} for a full account.

\subsection{Global exponential coordinates}
\label{sec:global_exponential_coordinates}
The graded vector bundle $\tHM = \bigoplus \tHM[i]$ is equipped with a family of \emph{dilations} $(\delta_\lambda)_{\lambda\in\RR}$, which are vector bundle endomorphisms such that $\delta_\lambda$ acts on $\tHM[i]$ by multiplication by $\lambda^i$.  
For $\lambda\neq0$, the dilations $\delta_\lambda$ are Lie algebra automorphisms in each fibre.
Therefore, they also define Lie group automorphisms in the fibres of $\THM$.

Let $\nabla$ be a connection on the vector bundle $\tHM$ that is compatible with the grading, in the sense that
\[ \delta_\lambda \circ \nabla \circ \delta_\lambda^{-1} = \nabla .\]
Equivalently, $\nabla$ is a direct sum of connections on each $\tHM[i]$.

\begin{definition}
 \label{def:splitting}
 A {\em splitting} of $\tHM$ is an isomorphism of vector bundles $\psi:\tHM \to TM$ such that the restriction of $\psi$ to $\tHM[i]$ is right inverse to the grading map $\sigma_i$ for each $i$.
\end{definition}

A choice of splitting $\psi$ allows us to transport the connection $\nabla$ to a connection
$\nabla^\psi = \psi\circ \nabla\circ \psi^{-1}$ on $TM$.
Associated to the  connection on $TM$ is the geometric exponential map $\exp^{\nabla^\psi}:TM \supseteq U \to M$, defined on some neighbourhood $U$ of the zero section.

Let $U\subseteq TM$ be an open neighbourhood of the zero section, such that the map
\begin{equation}
\label{eq:domain_of_injectivity}
 \Exp^{\nabla^\psi} : U \to M\times M; \quad  (x,\xi) \mapsto (x,\exp_x^{\nabla^\psi}(-\xi))
\end{equation}
is well-defined and injective.  
Consider the open neighbourhood
\begin{equation}
\label{eq:dom_inj_2}
  \tilde{\UU} := \{(x,\xi,t)\in \tHM\times\RR \st (x,\psi(\delta_t\xi))\in U\}
\end{equation}
of the zero section in $\tHM\times\RR$.  Then the map  
\begin{align}
\label{eq:EExp}
\EExp^{\psi,\nabla}\;:\;  \tHM \times \RR \; \supseteq \; \tilde{\UU} &\longrightarrow \THM \times\{0\} \;\sqcup\; M\times M \times \RR^\times; \\
  (x,\xi,t) &\longmapsto \begin{cases}
                        (\exp_x^{\nabla^\psi}(x, -\psi(\delta_t \xi)), t) &\text{if }t\neq0,\\
                        (x, \xi, 0)  &\text{if }t=0,
                     \end{cases} 
                     \nonumber
\end{align}
is well-defined and injective.  

In  \cite{VanYun:groupoid} we show that there is a smooth structure on $\TTHM$ such that the map $\EExp^{\psi,\nabla}$ is a diffeomorphism onto its image for any choice of graded connection $\nabla$, splitting $\psi$ and domain of injectivity $U$.
We emphasize that it is crucial that the exponential $\nabla$ is compatible with the grading.

\begin{definition}
 \label{def:global_exponential_coordinates}
 The image of $\tilde{\UU}$ under the exponential map is the set
 \begin{equation}
 \label{eq:exponential_patch}
 \UU = (\THM\times\{0\}) \sqcup (\Exp^{\nabla^\psi}(U) \times \RR^\times).
 \end{equation}
 Such a set will be called a {\em global exponential coordinate patch} for $\TTHM$.  
\end{definition}

These global exponential coordinate patches $\UU$ are open neighbourhoods of both the osculating groupoid $\THM\times\{0\}$ and the object space $\TTHM^{(0)} = M\times \RR$.


\section{Pseudodifferential operators}
\label{sec:PsiDOs}

\subsection{The zoom action on the tangent groupoid}
\label{sec:action}

The tangent groupoid comes  equipped with an action of $\RR^\times_+$.  It was of central importance in \cite{DebSka}.  

\begin{definition}
 \label{def:zoom}
 The {\em zoom action} $\zoom : \RR^\times_+ \to \Aut(\TTHM)$ is the smooth one-parameter family of Lie groupoid automorphisms defined as follows:
\begin{align}
 \label{eq:zoom}
  \zoom_\lambda (x,y,t) &= (x,y,\lambda^{-1}t) && \text{if } (x,y)\in M\times M,~t\neq0,\\
  \zoom_\lambda (x,\xi,0) &= (x, \delta_\lambda (\xi) , 0) && \text{if } (x,\xi) \in \THM,~t=0. \nonumber
\end{align}
\end{definition}

Using exponential coordinates, it is easy to see that the zoom action is smooth.
The zoom action $\tilde\zoom$ of $\RR^\times_+$ on $\tHM \times \RR$ is
\begin{equation}
 \label{eq:zoom_in_coords}
  \tilde\zoom_\lambda (x,\xi,t) = (x, \delta_\lambda (\xi), \lambda^{-1} t),
\end{equation}
which satisfies $\zoom_\lambda \circ \EExp^{\psi,\nabla} = \EExp^{\psi,\nabla} \circ \tilde\zoom_\lambda$.  
The zoom action induces a one-parameter family of automorphisms $(\zoom_{\lambda*})_{\lambda\in\RR^\times_+}$ on the convolution algebra $\sE_{r}'(\TTHM)$,
and preserves the subalgebra $\sE'_{r,s}(\TTHM)$ of proper $r$-fibred distributions, as well as the ideal $\Cp^\infty(\TTHM;\Omega_r)$ of properly supported smooth densities.

\subsection{Definition of pseudodifferential operators}
\label{sec:homogeneous}

We  now give our definition of pseudodifferential operators.  


\begin{definition}
\label{def:PPsiDOs}\label{def:cocycle}
A properly supported $r$-fibred distribution $\PP\in \sE_r'(\TTHM)$ is called {\em essentially homogeneous of weight $m\in\RR$} if 
\begin{equation}
\label{eq:essentially_homogeneous}
 \zoom_{\lambda*}\PP - \lambda^m \PP \in \Cp^\infty(\TTHM;\Omega_r)
   \qquad\qquad \text{for all $\lambda\in\RR^\times_+$.}
\end{equation}
The space of such distributions will be denoted $\PPsi^m_H(M)$.
\end{definition}

For $t\in\RR$, we will denote by $\PP_t$ the restriction of $\PP$ to the fibre $\TTHM|_t$, so
\begin{align*}
 \PP_t &\in \sE'_r(M\times M) \quad\text{for $t\neq0$},\\
 \PP_0 &\in \sE'_r(\THM).
\end{align*}

\begin{definition}
\label{def:PsiDOs}
A Schwartz kernel $P\in\sE_{r}'(M\times M)$ is  an \emph{$H$-pseudodifferential kernel} of order $\leq m$ if $P = \PP_1$ for some $\PP\in\PPsi^m_H(M)$. 
\end{definition}
The set of $H$-pseudodifferential kernels of order $\leq m$ will be denoted $\Psi_H^m(M)$.
We will generally blur the distinction between the kernel $P$ and the associated $H$-pseudodifferential operator $\Op(P)$,
\[
  \Op(P) : C^\infty(M) \to C^\infty(M); \quad
  (\Op(P)f)(x) = \int_{y\in M} P(x,y) f(y).
\]

\subsection{Cocycles}

The smooth differences \eqref{eq:essentially_homogeneous} will appear frequently in the ensuing analysis, so we shall give them a name.

\begin{definition}
The function $F:\RR^\times_+ \to \Cp^\infty(\TTHM;\Omega_r)$ defined by 
\begin{equation}
\label{eq:cocycle_def}
 F:\lambda \mapsto F_\lambda := \lambda^{-m} \zoom_{\lambda*}\PP -\PP
\end{equation}
will be called the {\em cocycle associated to $\PP\in\PPsi^m_H(M)$}.  
\end{definition}

We call $F_\lambda$ a cocycle because 

\begin{equation}
\label{eq:cocycle_identity}
 F_{\lambda_1\lambda_2} = \lambda_1^{-m} \zoom_{\lambda_1*} F_{\lambda_2} + F_{\lambda_1}.
\end{equation}
Our definition of $H$-pseudodifferential operators requires that the family $\PP_{\lambda t}-\lambda^m\PP_t$ is smooth in $t>0$ for fixed $\lambda>0$.
It will  be useful to know that it is a smooth function of $(t,\lambda)$ combined.

\begin{lemma}
 \label{lem:smooth_cocycle}
For any $\PP\in\PPsi_H^m(M)$, the associated cocycle $F : \lambda \mapsto \lambda^{-m}\zoom_{\lambda*}\PP -\PP$ is a smooth function of $\RR^\times_+$ into $\Cp^\infty(\TTHM;\Omega_r)$.  
\end{lemma}

\begin{proof}
Without loss of generality, we may assume that $M$ is compact.  The cocycle $F$ is certainly smooth as a map $\RR^\times_+ \to \sE'_r(\TTHM)$.  For any $k\geq1$, an application of the Arzela-Ascoli and Baire Category Theorems shows that,  on some open interval $(a,b)$, $F$ is bounded as a map into $C^k(\TTHM;\Omega_r)$ and so continuous into $C^{k-1}(\TTHM;\Omega_r)$ on $(a,b)$.  But the cocycle identity \eqref{eq:cocycle_identity} shows that the domain of continuity of $F$ into $C^{k-1}(\TTHM;\Omega_r)$ is $\RR^\times_+$-invariant.  Therefore $F:\RR^\times_+ \to C^\infty(\TTHM;\Omega_r)$ is continuous.

To prove that $F$ is {\em smooth}, 
we use a convolution trick on the multiplicative group $\RR^\times_+$.
Pick a compactly supported bump function $\phi \in \Cc^\infty(\RR^\times_+)$ with $\int_{\RR^\times_+} \phi(t) \frac{dt}{t} = 1$, and define $\QQ\in\sE'(\TTHM;\Omega_r)$ as an average of translates of $\PP$,
\[
  \QQ := \int_{\RR^\times_+} \phi(t) \, t^{-m} \alpha_{t*} \PP\,  \frac{dt}{t}.
\]
Then $ \QQ - \PP = \int_0^\infty \phi(t) \, F_t \, \frac{dt}{t} \in C^\infty(\TTHM;\Omega_s)$, thanks to the continuity of $F:\RR^\times_+ \to C^\infty(\TTHM;\Omega_r)$.
Also,
\begin{align*}
  (\lambda^{-m} \alpha_{\lambda*} \QQ) - \PP &= \int_0^\infty \phi(t) \,((\lambda t)^{-m} \alpha_{\lambda t*} \PP -\PP) \frac{dt}{t}\\
    &= \int_0^\infty \phi(\lambda ^{-1}t)\, F_t \frac{dt}{t},
\end{align*}
and by differentiation under the integral we see that this is a smooth function of $\lambda\in\RR^\times_+$ into $C^\infty(\TTHM;\Omega_s)$.  So, by writing 
\[
 F_\lambda = \lambda^{-m}\alpha_{\lambda*}(\PP-\QQ) + (\lambda^{-m} \alpha_{\lambda*} \QQ - \PP),
\]
we obtain that $F$ is smooth.
\end{proof}

\subsection{Pseudolocality}

Our definition of $\PPsi^m_H(M)$ forces the singular support of an $H$-pseudodifferential kernel to lie on the diagonal.

\begin{proposition}
\label{prop:smooth_off_the_units}
For every $\PP \in \PPsi^m_H(M)$ the restriction to $\TTHM \setminus \TTHM^{(0)}$ is a smooth density.  Hence every $P \in \Psi_H^m(M)$ is smooth off the diagonal.
\end{proposition}

\begin{proof}
The singular support of $\PP \in \PPsi^m_H(M)$ is invariant under the zoom action.  If $\gamma \in \TTHM \setminus \TTHM^{(0)}$ then its orbit under the zoom action is not proper.  Since $\PP$ has proper support by definition, the result follows. 
\end{proof}

\begin{remark}
In Section \ref{sec:conormality}, we will see that the definition also forces $H$-pseudodifferential kernels to be \emph{proper} fibred distributions, so they satisfy the equivalent conditions of Proposition \ref{prop:regular}.  
\end{remark}

\subsection{Examples}

\begin{example}
\label{ex:diff_op}
Returning to the example  from Section \ref{sec:outline}, we verify that differential operators are indeed in our calculus.  Here we will deal with the case of trivially filtered manifolds with $H^1=TM$.
For general filtered manifolds, see Section \ref{sec:diff_ops}.

Let $D$ be a differential operator $D$ of order $m$ on the (trivially filtered) manifold $M$ and let $P\in\sE'_r(M\times M)$ be its Schwartz kernel.  We define a family of $r$-fibred distributions $P_t$ on $\TTM|_t$ by
 \begin{equation}
 \label{eq:DO_family}
  \PP_t = \begin{cases}
           t^m P, & t\neq 0,\\
           \sigma_m(P), & t=0,
          \end{cases}
 \end{equation}
 where $\sigma_m(P)_x \in \sE'_r(T_xM)$ is the convolution kernel of the principal part of $D$, frozen at $x\in M$.  Since at $t\neq0$, the zoom action $\alpha_\lambda$ simply shifts the fibre $\TTHM|_t$ to $\TTHM|_{\lambda^{-1}t}$, $\PP|_{\RR^\times}$ is homogeneous on the nose of order $m$.  Also, $\sigma_m(P)_x$ is an order $m$ derivative of the Dirac distribution at $0\in T_xM$, so $\PP|_0$ is homogeneous of order%
\footnote{
There is a discrepancy here with the analysts' usual notion of order for derivatives of Dirac distributions, in which an order $m$ derivative of $\delta$ on $\RR^n$ has order $-n-m$.
The difference occurs because we are treating the Dirac distribution as a generalized \emph{density}, while analysts often use $\delta$ to refer to a generalized \emph{function}.  In that convention, the distribution we refer to as $\delta$ would instead correspond to $\delta\,|dx|$, where $|dx|$ denotes Lebesgue measure.  Since Lebesgue measure is homogeneous of order $n$, there is a shift of $n$ between the two conventions.  Moreover, we are defining homogeneity of distributions with respect to push-forward by dilations, not pull-back as one does for generalized functions, which accounts for the sign difference. 
} 
$m$.  The only nontrivial point to prove is that $\PP$ is a \emph{smooth} family, so defines an element of $\sE'_r(\TTHM)$.  We now prove this by explicit computation. A more conceptual explanation is given in Section \ref{sec:diff_ops}.
 
The Schwartz kernel of $D$ is 
\begin{equation}
\label{eq:DO_family1}
 P(x,y) = D^t\delta_x (y) ,
\end{equation}
as one sees from the calculation 
\[
\int_{y\in M} D^t\delta_x(y)\, f(y) = \int_{y\in M} \delta_x(y)\, Df(y) = Df(x).
\]
Let us write $D^t$ in coordinates as
\[
 D^t=\sum_{|\bfa|\leq m} c_\bfa(x) \frac{\partial^\bfa}{\partial x^\bfa},
\]
where $\bfa = (a_1,\ldots,a_n) \in \NN^n$ are multi-indices and $c_\bfa(x)$ are smooth coefficient functions.  Then 
\[
 P(x,y) = \sum_{|\bfa|\leq m} c_\bfa(x) \delta_x^{(\bfa)}(y),
\]
where $\delta_x^{(\bfa)}$ denotes the $\bfa$-derivative of the Dirac distribution supported at $x$.

The principal part of $D$, frozen at $x$, is the differential operator on $T_xM$ given by
\[
 D_x = \left( \sum_{|\bfa|= m} c_\bfa(x) \frac{\partial^\bfa}{\partial \xi^\bfa} \right)^t 
    = \sum_{|\bfa|= m} (-1)^m c_\bfa(x) \frac{\partial^\bfa}{\partial \xi^\bfa},
\]
where $\xi$ denotes the canonical coordinates on $T_x\RR^n \cong \RR^n$.   It is given by convolution with the distribution
\[
 \sigma_m(P)(x,\xi) := \sum_{|\bfa|= m} (-1)^m c_\bfa(x) \delta_0^{(\bfa)}(\xi) 
   = \sum_{|\bfa|= m} c_\bfa(x) \delta_0^{(\bfa)}(-\xi),
\]
where $\delta_0^{(\bfa)}$ denotes the $\bfa$-partial derivative of $\delta_0$

Using the standard connection $\nabla$ on $\RR^n$ and splitting $\psi=\id$, the global exponential map of Section \ref{sec:global_exponential_coordinates} becomes
\begin{align}
\label{eq:Rn_exponential}
  \EExp^{\nabla}  : T(\RR^n) \times \RR &\stackrel{\cong}{\longrightarrow} \TT(\RR^n) \\
  (x,\xi,t) &\longmapsto  (x, x-t \xi, t) &&\text{if }t\neq0, \nonumber \\
  (x,\xi,0) &\longmapsto  (x, \xi, 0)  &&\text{if }t=0,                     
  \nonumber
\end{align}
Writing $\PP_t$ in these coordinates (for $t\neq0$) yields the distribution
\[
 \PP_t(x,\xi) =  t^m \sum_{|\bfa|\leq m} c_\bfa(x) \delta^{(\bfa)}(-t\xi) 
   = \sum_{|\bfa|\leq m} t^{m-|\bfa|} c_\bfa(x) \delta^{(\bfa)}(-\xi)
\]
on $T(\RR^n) \times \{t\}$.  As claimed, this extends to a smooth family of distributions on $TM\times\RR$ with $\sigma_m(P)(x,\xi)$ at $t=0$.
Thus $P\in \Psi^m_H(M)$.
\end{example}

\begin{example}
\label{ex:identity}
On any filtered manifold, the identity operator $I:C^\infty(M) \to C^\infty(M)$ is an $H$-pseudodifferential operator of order $0$. 
An extension $\II\in\PPsi^0_H(M)$ is the $r$-fibred distribution $\II \in \sE'_r(\TTHM)$ given by restriction to the unit space: $(\II,f) = f|_{M\times\RR}$.

More generally, multiplication by a smooth function $a\in C^\infty(M)$ is $H$-pseudodifferential of order $0$.  An extension is given by $\mathbbl{a}\II\in\PPsi^0_H(M)$, where $\mathbbl{a}\in C^\infty(M\times \RR)$ is defined by $\mathbbl{a}(x,t) = a(x)$.
\end{example}

\begin{example}
\label{ex:logarithmic_singularity}

Typically, in a kernel-based approach to pseudodifferential calculi, logarithmic singularities need to be included explicitly in the asymptotic expansion (e.g.\ \cite{ChrGelGloPol}). 
One of the advantages of the present approach is that they enter naturally. Compare also \cite[Section I.2]{Taylor:microlocal}.

To illustrate this by an example, let $M= \RR^n$ with the trivial filtration.  
Using  the exponential coordinate system \eqref{eq:Rn_exponential}, consider the fibred distribution $\PP\in \sE'_r(M)$ whose pull-back $\PPtilde:= \PP\circ \EExp^\nabla$ is
 \[
  \PPtilde(x,\xi,t) = \log|\xi|\,d\xi.
 \]
 In standard coordinates,
 \begin{align*}
  \PP(x,y,t) & = \log (|t|^{-1}|x-y|) \,dy, &\qquad& \text{if }t\neq 0, \\
  \PP(x,\xi,t) & = \log |\xi| \, d\xi, &\qquad& \text{if }t\neq 0. 
 \end{align*}
Then
 \[
  \alpha_{\lambda*}\PPtilde(x,\xi,t) = \log|\lambda^{-1}\xi|\,d(\lambda^{-1}\xi) 
   = \lambda^{-n}(\log|\xi|-\log|\lambda|)\,d\xi,
 \]
 so
 \[
  (\alpha_{\lambda*}\PPtilde - \lambda^{-n}\PPtilde)(x,\xi,t) = -\lambda^{-n}\log|\lambda|\,d\xi
 \]
 which is a smooth density on $T(\RR^n)\times\RR$ for any $\lambda\in\RR^\times_+$.  Thus $\PP$ is essentially homogeneous
of weight $-n$.  It is not properly supported, but one can easily resolve this by multiplying $\PP$ by an exponential cut-off function (see the next section).  Thus, after a properly supported cut-off, $\PP_1(x,y) = \log|x-y|\,dy$ becomes an $H$-pseudodifferential kernel.
\end{example}


\section{Local exponential coordinates}

Our definition of $H$-pseudodifferential operators is coordinate independent.
Still, in order to prove  various properties of these operators, we need coordinates on $\TTHM$.
In this section we introduce a convenient set of coordinates, and fix notation for the rest of the paper.

\subsection{Exponential cut-offs}
\label{sec:cut-offs}

The essential data of an $H$-pseudodifferential operator is carried on a neighbourhood of the unit space $\TTHM^{(0)} = M\times \RR$ and the osculating groupoid $\TTHM|_0 = \THM$ inside $\TTHM$.  This can be made precise as follows.

We assume that we have fixed, once and for all, a choice of graded connection $\nabla$, splitting $\psi$ and domain of injectivity $U\subseteq \tHM$ as in Section \ref{sec:global_exponential_coordinates}.
Recall from Section \ref{sec:global_exponential_coordinates} that the global exponential coordinate patch
  \[
   \UU := (\THM \times \{0\}) \:\sqcup\: (\Exp^{\nabla^\psi}(U) \times \RR^\times) \quad \subseteq \TTHM.
  \]
 is an open neighbourhood of $\TTHM^{(0)} \cup \TTHM|_0$.
\begin{lemma}
\label{lem:exponential_support}

For any $\PP\in\PPsi_H^m(M)$ one can find $\PP'\in\PPsi^m_H(M)$ that is supported in $\UU$ and such that $\PP-\PP' \in \Cp^\infty(\TTHM;\Omega_r)$.
\end{lemma}

\begin{proof}
Let $\varphi_1 \in C^\infty(M\times M)$ be a smooth bump function that  equals $1$ on a neighbourhood of the diagonal and zero outside  $U$.  The smooth function $\varphi$ on $\TTHM$ defined by
\begin{equation}
\label{eq:exponential_cut-off}
 \begin{cases}
   \varphi(x,y,t) = \varphi_1(x,y), &\quad \text{if }t\neq0, \\
   \varphi(x,\xi,0) = 1     &\quad \text{if }t=0.
 \end{cases}
\end{equation}
is constant $1$ on a neighbourhood of $\TTHM^{(0)} \cup \TTHM|_0$, and $0$ outside of $\UU$.  It is also invariant under the zoom action.   Therefore, putting $\PP' := \varphi \PP$ we obtain the result.
\end{proof} 
 
Since elements of $\Cp^\infty(\TTHM;\Omega_r)$ are trivially in $\PPsi^m_H(M)$ for any $m$, this lemma will allow us to routinely assume that $\PP$ is supported in an exponential coordinate patch.

\begin{definition}
\label{def:global_pullback}
If $\PP\in\PPsi^m_H(M)$ is supported in the global exponential coordinate chart $\UU$, as in Lemma \ref{lem:exponential_support}, we  write $\widetilde\PP\in\sE_r'(\tHM\times\RR)$ for the pullback of $\PP$ via $\EExp^{\psi,\nabla}$.
\end{definition}


\subsection{Local exponential coordinates}
\label{sec:charts}

Let $\RR^n \supset M_0 \into M$ be a coordinate chart for $M$ with compact closure upon which all of the bundles $H^i$ are trivializable.  We identify $M_0$ with its image in $M$.  
We thus have a trivialization of graded vector bundles
\begin{equation}
\label{eq:tHM_trivialization}
 M_0 \times V \xrightarrow{\;\cong\;} \tHM|_{M_0},
\end{equation}
where $V$ is a graded vector space.  

The exponential map $\EExp^{\psi,\nabla}$  from Section \ref{sec:global_exponential_coordinates} defines a diffeomorphism from an open set $\tilde{\UU} \subseteq \tHM\times\RR$ to an open set $\UU \subseteq \TTHM$.  Let us write $\VV \subseteq M_0\times V \times \RR$ for the preimage of $\tilde{\UU}|_{M_0\times\RR}$ under the trivialization \eqref{eq:tHM_trivialization}.

\begin{definition}
\label{def:local_exponential_coordinates}
With the notation as above, the composition
\begin{equation}
 \label{eq:trivialization}
 \xymatrix{
 \quad \VV \quad \ar[r]_\cong \ar@{}[d]|-*[@]{\subset}
   & \quad \tilde{\UU} \quad \ar[r]^-{\EExp^{\psi,\nabla}}_-\cong \ar@{}[d]|-*[@]{\subset}
   & \quad \UU \quad \ar@{}[d]|-*[@]{\subset} \\
 M_0 \times V\times \RR  & \tHM|_{M_0} \times\RR  & \TTHM,
 }
\end{equation}
will be called a {\em local exponential coordinate chart}, and its image in $\TTHM$ (a subset of $\UU$) will be called a \emph{local exponential coordinate patch}.
\end{definition}

\begin{lemma}
\label{lem:local_exponential_support}
  Any $\PP\in\PPsi^m(M)$ can be decomposed as a locally finite sum $\PP = \sum_i \PP_i + \QQ$ where each $\PP_i \in \Psi^m_H(M)$ is supported in some local exponential coordinate patch as in Definition \ref{def:local_exponential_coordinates}, and $\QQ\in\Cp^\infty(\TTHM;\Omega_r)$.
\end{lemma}

\begin{proof}
Use Lemma \ref{lem:exponential_support}  and a partition of unity.
\end{proof}


\subsection{Notation for graded vector spaces}
\label{sec:graded_coordinates}

The following terminology for graded vector spaces $V = \bigoplus_i V_i$ will be used throughout the paper.

A {\em graded basis} for $V$ is a basis $\{e_1, \ldots, e_n\}$ such that each element $e_j$ is contained in some $V_i$.  The corresponding linear coordinates $\xi = (\xi_1,\ldots, \xi_n)$ will be called a {\em graded coordinate system}.  We will write $d_j$ for the degree of $e_j$, meaning $e_j \in V_{d_j}$. The {\em homogeneous dimension} of $V$ is $d_H = \sum_{j=1}^n d_j$.  If $\mathbf{a} = (a_1,\ldots,a_n) \in \NN^n$ is a multi-index 
we write $|\mathbf{a}| = \sum_j a_j$ for its {\em usual order} and $|\mathbf{a}|_H = \sum_j a_jd_j$ for its {\em homogeneous order}.

Now suppose we have a local exponential coordinate chart $M_0 \times V \times \RR \supseteq \VV \to \UU \subseteq \TTHM$ as in Definition \ref{def:local_exponential_coordinates}.  Choose coordinates $x=(x_i)$ for $M_0$, graded coordinates $\xi=(\xi_i)$ for $V$, and let $t$ be the standard coordinate on $\RR$.  The system $(x,\xi,t)$ will be referred to as {\em local exponential coordinates}.  We will reuse $r$ to denote the bundle projection
\[
  r: M_0\times V \times \RR \to M_0\times \RR
\]
which corresponds to the range projection of $\TTHM$ under the exponential map \eqref{eq:trivialization}.

\begin{definition}
\label{def:local_pullback}
If $\PP\in\PPsi^m_H(M)$ is supported in a local exponential coordinate patch, then we will write $\PPtilde(x,\xi,t)$ to denote its pullback  $\PPtilde \in \sE'_r(M_0\times V \times \RR)$ via the local exponential coordinates \eqref{eq:trivialization}.
\end{definition}

This pullback $\PPtilde(x,\xi,t)$ is a smooth family of distributions in $\xi$, indexed by $(x,t)\in M_0\times\RR$.  The double use of the notation $\PPtilde\in\sE'_r(\tHM\times\RR)$ for pullbacks via global exponential coordinates (Definition \ref{def:global_pullback}) and $\PPtilde\in\sE'_r(M_0 \times V \times \RR)$ for pullbacks via local exponential coordinates (Definition \ref{def:local_pullback}) should not cause confusion---the appropriate variant will be made clear when necessary.


\section{Principal cosymbols}
\label{sec:cosymbols}

In this section we show that the highest order part of an $H$-pseudodifferential operator can be identified with a distribution on $\THM$.

\begin{lemma}
\label{lem:smooth_cosymbol}
Let $\PP \in \PPsi^m_H(M)$.  If $\PP_1 \in \Cp^\infty(M\times M;\Omega_r)$ then $\PP_0 \in \Cp^\infty(\THM)$.
\end{lemma}

\begin{proof}
Essential homogeneity immediately implies that $\PP_t \in \Cp^\infty(M\times M;\Omega_r)$ for all $t\neq0$.  We use an Arzela-Ascoli  argument to show that $\PP_0$, as the limit of the sequence $\PP_1,\PP_\frac{1}{2}, \PP_\frac{1}{4}, \PP_\frac{1}{8}, \dots$  is also smooth.

By Lemma \ref{lem:local_exponential_support}, we may assume that $\PP$ is supported in a local exponential coordinate patch, and we will work with  $\widetilde\PP \in \sE'_r(M_0\times V \times \RR)$, as in Definition \ref{def:local_pullback}. 

Let $\omega = (\omega_{x,t})_{(x,t)\in M_0\times\RR}$ denote the constant family of Lebesgue measures on $V$.
Then  $\widetilde{\PP}|_{M_0\times V \times \RR^\times} = p \omega$ for some $p \in C^\infty(M_0 \times V \times \RR^\times)$. Likewise, we write the cocycle associated to $\widetilde{\PP}$ as $F_\lambda = f_\lambda \omega$ where $f_\lambda \in C^\infty(M_0 \times V \times \RR)$.  

The zoom action acts on fibrewise Lebesgue measure by $\tilde\zoom_{\lambda*} \omega = \lambda^{-d_H} \omega$.  It follows that
\begin{equation}
\label{eq:function_cocycle}
 f_\lambda = \lambda^{-m-d_H} \tilde\zoom_{\lambda^{-1}}^* p - p.
\end{equation}
Differentiating this gives
\[
 \partial_x^\bfa \partial_\xi^\bfb f_\lambda(x,\xi,t) = \lambda^{-m-d_H-|\bfb|_H} (\partial_x^\bfa \partial_\xi^\bfb p)(x,\delta_{\lambda^{-1}}(\xi),\lambda t) - (\partial_x^\bfa \partial_\xi^\bfb p)(x,\xi,t)
\]
for any multi-indices $\mathbf{a}, \mathbf{b} \in \NN^n$.  By fixing $\lambda=\half$, we get the following estimate in uniform norms for every $t\in[0,1]$:
\begin{multline*}
 \left| \; 2^{m+d_H+|\mathbf{b}|_H} \| \partial_x^\mathbf{a} \partial_\xi^\mathbf{b} p|_{\frac{1}{2} t} \|_{\overline{M_0}\times V \times \{\half t\}} - 
    \|\partial_x^\mathbf{a} \partial_\xi^\mathbf{b} p|_t \|_{\overline{M_0}\times V \times \{t\}} \; \right|
   \\
  \leq  \| \partial_x^\mathbf{a} \partial_\xi^\mathbf{b} f_{\frac12} \|_{\overline{M_0}\times V\times[0,1]},
\end{multline*}
where $p|_t$ denotes the restriction of $p$ to $M\times V \times \{t\}$.

Let us put $m'=m+d_H+|\mathbf{b}|_H$.  We obtain
\begin{align*}
 \lefteqn{ \| \partial_x^\mathbf{a} \partial_\xi^\mathbf{b} p|_{2^{-k}} \|_{\overline{M_0}\times V}} \qquad \\
  &= \left( \| \partial_x^\mathbf{a} \partial_\xi^\mathbf{b} p|_{2^{-k}} \|_{\overline{M_0}\times V} - 
   2^{-m'} \| \partial_x^\mathbf{a} \partial_\xi^\mathbf{b} p|_{2^{-(k-1)}} \|_{\overline{M_0}\times V} \right) \\
  &\qquad + \left( 2^{-m'} \| \partial_x^\mathbf{a} \partial_\xi^\mathbf{b} p|_{2^{-(k-1)}} \|_{\overline{M_0}\times V} 
    - 2^{-2m'} \| \partial_x^\mathbf{a} \partial_\xi^\mathbf{b} p|_{2^{-(k-2)}} \|_{\overline{M_0}\times V}\right)  \\
  & \qquad +\cdots + \left( 2^{-(k-1)m'} \| \partial_x^\mathbf{a} \partial_\xi^\mathbf{b} p|_{2^{-1}} \|_{\overline{M_0}\times V} 
    - 2^{-km'} \| \partial_x^\mathbf{a} \partial_\xi^\mathbf{b} p|_{1} \|_{\overline{M_0}\times V}\right) \\
  & \qquad  + 2^{-km'} \| \partial_x^\mathbf{a} \partial_\xi^\mathbf{b} p|_{1} \|_{\overline{M_0}\times V} \\
  & \leq \left( \sum_{j=1}^k 2^{-jm'} \right) \| \partial_x^\mathbf{a} \partial_\xi^\mathbf{b} f_{\half} \|_{\overline{M_0}\times V\times[0,1]} + 2^{-km'} \| \partial_x^\mathbf{a} \partial_\xi^\mathbf{b} p|_{1} \|_{\overline{M_0}\times V}.
\end{align*}
for every $k\in\NN$.
If $m' >0$, this is bounded as $k\to\infty$.  That is, for all multi-indices $\mathbf{a},\mathbf{b}$ with $|\mathbf{b}|_H > -m-d_H$, the sequence of functions $(p|_{2^{-k}})_{k\in\NN}$ is bounded with respect to the seminorm $ \| \partial_x^\mathbf{a} \partial_\xi^\mathbf{b} \,\bullet\, \|_{\overline{M_0}\times V}$.  Since $\supp(p) \cap (\overline{M_0} \times V \times[0,1])$ is compact, the derivatives $\partial_x^\mathbf{a} \partial_\xi^\mathbf{b} p$ with $|\mathbf{b}|_H \leq -m-d$ are bounded in terms of those with  $|\mathbf{b}|_H > -m-d_H$.  By the Arzela-Ascoli Theorem, it admits a limit point $p|_0 \in C^\infty(\overline{M_0}\times V)$.  But $\PPtilde_{2^{-k}}\to \PPtilde_0$ as distributions, so $\PPtilde_0 = p|_0 \omega$. 
\end{proof}

\begin{corollary}
\label{cor:principal_symbol}
Let $\PP, \QQ \in \PPsi^m_H(M)$. If $\PP_1=\QQ_1$, 
then $\PP_0 \equiv \QQ_0$ modulo $\Cp^\infty(\THM;\Omega_r)$.
\end{corollary}

\begin{definition}
 \label{def:cosymbol_space}
For $m\in\RR$, we define
\begin{align*}
  \Sigma_H^m(M) := \{ K \in \sE_{r}'(\THM) /& \Cp^\infty(\THM;\Omega_r) \st \\
  & \delta_{\lambda*}K = \lambda^m K  \text{ for all } \lambda\in\RR^\times_+ \} .
\end{align*}
Elements of $\Sigma_H^m(M)$ will be called {\em cosymbols of order $m$} on $M$.
\end{definition}


\begin{definition}
\label{def:cosymbol_map}
We define the {\em cosymbol map} 
\[
\sigma_m : \Psi^m_H(M) \to \Sigma^m_H(M)
\]
as follows.  For $P\in\Psi^m_H(M)$, choose any $\PP\in\PPsi^m_H(M)$ with $\PP_1 = P$ and set $\sigma_m(P) := \PP_0$ modulo $\Cp^\infty(\THM;\Omega_r)$.  
\end{definition}

This is well-defined by Corollary \ref{cor:principal_symbol}.

\begin{lemma}
 \label{lem:cosymbols_extend}
The cosymbol map $\sigma_m$ is surjective.
\end{lemma}

\begin{proof}
Consider a distribution $K \in \sE_r'(\THM)$ representing a cosymbol of order $m$.
Define the constant family $\widetilde{\KK} \in \sE'_r(\tHM\times\RR)$ by $\widetilde{\KK}_t = K$ for all $t\in\RR$.  Then $\widetilde{\KK}$ is essentially homogeneous for the action $(\tilde{\zoom}_\lambda)$ of Equation \eqref{eq:zoom_in_coords}.  This can be transferred to an essentially homogeneous distribution $\KK$ on $\TTHM$ by use of an exponential cut-off (Section \ref{sec:cut-offs}).  It has $\KK_0 = K$.
\end{proof}

\begin{proposition}
\label{prop:vanishing}
 Let $\PP \in \PPsi^m_H(M)$, $k\in\ZZ$.  The following are equivalent:
 \begin{enumerate}
  \item $\PP = t^k \QQ$ for some $\QQ \in \sE'_r(\TTHM)$,
  \item the pullback $\PPtilde \in \sE'_r(\tHM\times\RR)$ of $\PP$ under any global exponential coordinates satisfies $(\partial_t^j\PP)|_{t=0} = 0$ for $j=0,\ldots,k-1$.
 \end{enumerate}
 In this case, $\QQ \in \PPsi^{m-k}_H(M)$.
\end{proposition}

\begin{proof}
Since $\PPtilde \in \sE'_r(\tHM \times \RR) \cong C^\infty(\RR, \sE'_r(\tHM))$, the equivalence of the vanishing conditions $(1)$ and $(2)$ is immediate.  
The cocycle $\widetilde{F}$ associated to $\PPtilde$ also vanishes to order $k-1$ in $t$, so $F_\lambda = t^k G_\lambda$ for some $G_\lambda \in \Dens$.  For all $t\neq 0$ we have 
\begin{multline*}
(\lambda^{-m+k}\zoom_{\lambda*} \QQ - \QQ)_t = \lambda^{-m+k} \QQ_{\lambda t} - \QQ_t \\
 = t^{-k}(\lambda^{-m} \PP_{\lambda t} -\PP_t) = t^{-k}(\lambda^{-m}\zoom_{\lambda*} \PP - \PP)_t = G_\lambda|_t.
\end{multline*}
By continuity, we also get $(\lambda^{-m+k}\zoom_{\lambda*} \QQ - \QQ)_0 = G_\lambda|_0$.  Thus $\QQ \in \PPsi_H^{m-k}(M)$.
\end{proof}

\begin{corollary}
 \label{cor:lower_order}
 $\Psi^{m-1}_H(M) = \ker \sigma_m \subseteq \Psi^{m}_H(M)$ for every $m\in\ZZ$.
\end{corollary}

\begin{proof}
 The map $\PPsi^{m-1}_H(M) \into \PPsi^m_H(M);~\PP \mapsto t\PP$ yields the inclusion $\Psi^{m-1}_H(M) \into \ker \sigma_m$.  It is surjective by Proposition \ref{prop:vanishing}.
\end{proof}

This proves the short exact sequence \eqref{eq:short_exact_sequence} of the introduction,
and shows that the principal cosymbol $\sigma_m(P)\in \Sigma^m_H(M)$ is indeed the highest order part of $P\in  \Psi^{m}_H(M)$.

\section{Full symbols}
\label{sec:symbols}

Although Fourier transforms---"symbols"---play no role in our definition of pseudodifferential operators, they are (unsurprisingly) important for proving the main analytic properties.
We begin by recalling a few facts about vector bundle Fourier transforms.


\subsection{Fourier transforms on vector bundles}

Let $\pi:E \to M$ be a vector bundle and $\pi':E^*\to M$ its dual bundle.

Recall (see e.g.\ \cite{Carrillo-Rouse, DebSka}) that there is a well-defined space $\sS_\pi(E)$ of \emph{fibre-wise Schwartz class functions} on $E$, with the following Fr\'echet seminorms:
\begin{equation}
 \label{eq:Schwarz_seminorm}
 \| f \|_{K,\mathbf{a},\mathbf{b},\mathbf{c}} = \sup_{x\in K, \xi\in V} |\xi^\bfa \partial_\xi^\bfb \partial_x^\bfc f(x,\xi)|,
\end{equation}
where $(x,\xi) \in M_0\times V \into E$ are linear bundle coordinates, $K\subset M_0$ is a compact subset, and $\bfa,\bfb,\bfc\in\NN^n$ are multi-indices.  One can likewise define the \emph{Schwartz class vertical densities} $\sS_\pi(E;\Omega_\pi)$.

The \emph{fibrewise Fourier transform} $\sF_\pi: \sS_\pi(E;\Omega_\pi) \to \sS_{\pi'}(E^*)$ is defined by $u\mapsto \hat{u}$ where
\begin{equation}
 \label{eq:fibrewise_Fourier_transform}
  \hat{u}(\eta) = (u_x, e^{-i(\eta,\slot)}), \qquad  \qquad \eta\in E^*_x,~x\in M,
\end{equation}
This extends to a map between the  \emph{tempered fibred distributions}\linebreak $\sS'_\pi(E) \to \sS'_{\pi'}(E^*;\Omega_\pi^*)$, which are the $C^\infty(M)$-linear duals of the Schwartz spaces.  
In analogy with Definition \ref{Def:fibred_dist}, let $\sE'_\pi(E)$ be the set of $C^\infty(M)$-linear maps $u:C^\infty(E) \to C^\infty(M)$.
Note that $\sE'_\pi(E)$ is a subset of $\sS'_\pi(E)$.  The Fourier transform maps $\sE'_\pi(E)$ into $C^\infty(E^*)$.


\subsection{Full symbols}
\label{sec:full_symbols}

In the following definition, we view $\tHM \times \RR$ as a bundle over $M \times \RR$.  The fibrewise Fourier transform maps $\sE'_r(\tHM \times \RR)$ to $C^\infty(\tsHM\times\RR)$.

\begin{definition}
\label{def:PPhat}
Suppose $\PP\in\PPsi^m_H(M)$ is supported in a global exponential coordinate patch $\UU$ (see Equation \eqref{eq:exponential_patch}) and let  $\PPtilde\in\sE'_r(\tHM\times\RR)$ be its pullback under $\EExp^{\psi,\nabla}$.  We will write $\hat{\PP} \in C^\infty(\tsHM\times\RR)$ for the Fourier transform of $\PPtilde$.  

The \emph{full symbol} of the $H$-pseudodifferential operator $P = \PP_1$ is $\hat{P} = \hat{\PP}_1$.
\end{definition}

It is understood that any statements we make regarding $\hat{\PP}$ or $\hat{P}$ will hold regardless of the choice of exponential coordinates and cut-off.

\begin{example}
 \label{ex:full_symbol}
Consider the case of $M = \RR^n$ with trivial filtration. 
Using  the exponential coordinates \eqref{eq:Rn_exponential} we have $\widetilde{P}(x,\xi) = P(x,x-\xi)$.  Therefore, the full symbol $\hat{P}$ is defined by exactly the same formula that converts a classical pseudodifferential kernel to its symbol---see Equation (18.1.7) of \cite{Hormander:III}.
\end{example}

\subsection{Homogeneity of the Fourier transform}
\label{sec:Fourier_homogeneity}

Roughly speaking, the Fourier transform converts homogeneity of a distribution $u$ near zero into homogeneity of $\hat{u}$ at infinity.  In this section, we give a precise application of this principle to $\PPsi^m_H(M)$.  


Let $E = \bigoplus_i E_i$ be a graded vector bundle.  The dual bundle $E^*$ is also graded, thanks to the canonical decomposition $E^* \cong \bigoplus_i E_i^*$.  We denote by $(\delta_\lambda)$ and $(\delta'_\lambda)$ the canonical families of dilations on $E$ and $E^*$, respectively (see Section \ref{sec:global_exponential_coordinates}).   
Note that the fibrewise Fourier transform $\sF_\pi:\sE'_\pi(E) \to C^\infty(E^*)$ is compatible with the dilations:
\[
 \qquad \widehat{\delta_{\lambda_*} u} = {\delta'_\lambda}^* \hat{u}, \qquad\qquad \text{for all $u\in \sE'_\pi(E) $}.
\]

We apply this to $\PP\in\PPsi^m_H(M)$.  
From Equation \eqref{eq:zoom_in_coords}, the zoom action on $\PPtilde$ in exponential coordinates is given by
\[
 (\tilde{\alpha}_{\lambda*}\PPtilde)|_t = \delta_{\lambda*}(\PPtilde_{\lambda t}), \qquad\qquad t\in \RR, \lambda\in\RR^\times.
\]
Taking fibrewise Fourier transforms we get
\begin{align*}
  (\widehat{\zoom_{\lambda*} \PP}) (x,\eta,t) 
    &= (\delta_{\lambda*}\PPtilde_{(x,\lambda t)} \, , \, e^{-i(\eta,\slot)} ) \\
    &\qquad= ( \PPtilde_{(x,\lambda t)} \, , \, e^{-i(\delta'_\lambda(\eta),\slot)} ) 
    = \hat{\PP}(x, \delta'_\lambda\eta,\lambda t),
\end{align*}
where $(x,\eta,t) \in \tsHM \times \RR$.
That is, 
\begin{equation}
\widehat{\tilde\zoom_{\lambda*} \PPtilde} = \beta_\lambda^*\hat{\PP},
\end{equation}
 where $\beta_\lambda:\tsHM\times\RR \to \tsHM\times\RR$ is defined by
\begin{equation}
\label{eq:dual_zoom}
  \beta_\lambda(x,\eta,t) = (x, \delta'_\lambda(\eta),\lambda t).
\end{equation}
Therefore, the essential homogeneity of $\PP\in\PPsi^m_H(M)$ gives
\begin{equation}
\label{eq:homogeneity_of_Fourier_transform1}
  \beta_\lambda^*\hat{\PP} - \lambda^m \hat{\PP} \in \sS_{r'}(\tsHM\times\RR) \qquad \text{for all $\lambda\in\RR^\times_+$.}
\end{equation}

Importantly, the maps $\beta_\lambda$ are themselves dilations associated to a graded vector bundle.  Namely, we change our point of view to consider $\tsHM \times \RR$ as a vector bundle over $M$ (rather than $M\times \RR$) via the bundle projection 
\begin{equation}
 \label{eq:augmented_bundle}
 \rho: \tsHM \times \RR \to M ; \qquad (x,\eta,t) \mapsto x. 
\end{equation}
Then $(\beta_\lambda)$ is the family of dilations on this bundle, where $\tsHM$ has its original grading and the complementary fibre $\RR$  is given grading degree $1$.  

\begin{definition}
 \label{def:homogeneous_on_the_nose}
 An essentially homogeneous $r$-fibred distribution $\PP\in\PPsi^m_H(M)$ is called \emph{homogeneous on the nose outside $[-1,1]$} if 
for all $t>1$ we have $\PP_t=t^m\PP_1$ and $\PP_{-t}=t^m\PP_{-1}$.
\end{definition}

\begin{proposition}
 \label{prop:homogeneous_on_the_nose}
 Any $H$-pseudodifferential operator $P$ admits an extension to $\PP\in\PPsi^m_H(M)$ that is homogeneous on the nose outside $[-1,1]$.
\end{proposition}

\begin{proof}
Choose any $\QQ\in \PPsi^m_H(M)$ with $\QQ_1=P$.  Fix a smooth partition of unity $\phi_-$, $\phi_0$, $\phi_+: \RR\to[0,\infty)$ on $\RR$ with supports in $(-\infty,-\half)$, $(-1,1)$, $(\half,\infty)$, respectively, and define $\PP\in \sE'_r(\TTHM)$ by
\[
 \PP_t := \phi_-(t)\, t^m\QQ_{-1} + \phi_0(t)\, \QQ + \phi_+(t)\, t^m \QQ_{1}.
\]
Lemma \ref{lem:smooth_cocycle} implies that the family $t^m\QQ_1-\QQ_t$ is a smooth function of $t\in \RR^\times_+$ into $C^\infty_P(M\times M,\Omega_r)$, and similarly for $t^m\QQ_{-1}-\QQ_{-t}$ when $t<0$.  We obtain that $\PP-\QQ\in C^\infty_p(\TTHM;\Omega_r)$ and the result follows.
\end{proof}

Note that exponential cut-off functions are also homogeneous on the nose, so we may simultaneously arrange that $\PP$ be supported in a global exponential coordinate patch.

Returning now to the homogeneity of $\hat\PP$, let $\PP\in\PPsi^m_H(M)$ be homogeneous on the nose outside $[-1,1]$. Then the differences of Equation \eqref{eq:homogeneity_of_Fourier_transform1} have support bounded in $t$.  We obtain the improved condition:
\begin{equation}
\label{eq:homogeneity_of_Fourier_transform2}
  \beta_\lambda^*\hat{\PP} - \lambda^m \hat{\PP} \in \sS_{\rho}(\tsHM\times\RR) \qquad \text{for all $\lambda\in\RR^\times_+$.}
\end{equation}

Now, a bundle analogue of \cite[Proposition 2.2]{Taylor:microlocal} immediately gives the following.

\begin{proposition}
\label{prop:symbol_homogeneity}
 Let $\PP \in \PPsi^m_H(M)$ be homogeneous on the nose outside $[-1,1]$, and let $\hat{\PP}\in C^\infty(\tsHM\times\RR)$ be its Fourier transform (with respect to any exponential coordinates).  Then $\hat\PP$ is equal, modulo $\sS_\rho(\tsHM\times\RR)$, to a genuinely homogeneous function at infinity.   More precisely, there exists a smooth function $A \in C^\infty(\tsHM \setminus (M\times \{0\}\times\{0\}))$ satisfying
 \[
   \beta_\lambda^* A = \lambda^m A  \qquad \text{for all } \lambda\in\RR^\times_+,
 \]
and such that, for any bump function $\varphi\in C^\infty(\tsHM\times\RR)$ that is $0$ in some neighbourhood of the zero section of $\rho$ and $1$ outside some $\rho$-proper set, we have $\hat\PP - \varphi A \in \sS_\rho(\tsHM\times\RR)$.
\end{proposition}

\begin{remark}
\label{rmk:symbol_classes}
Symbols of order $m$ are usually defined in terms of the symbol inequalities
 \begin{equation}
 \label{eq:symbol_inequality}
  | \partial_x^\bfa \partial_\eta^\bfb a(x,\eta) | \leq C (1 + \|\eta\|)^{m-|\bfb|}, \qquad (x,\eta)\in T^*M.
 \end{equation}
The import of Proposition \ref{prop:symbol_homogeneity} is that, since the extension of $a=\hat{P}$ to $\hat{\PP} \in C^\infty(T^*M \times \RR)$ is homogeneous at infinity with respect to the  \emph{augmented} dilations $(\beta_\lambda)$, the symbol inequalities \eqref{eq:symbol_inequality} are automatically satisfied for $a$; see Corollary \ref{cor:decay_conditions} below.

 This simple observation explains the economy that is gained by passing to the tangent groupoid.  For instance, in the case of the Heisenberg calculus, a great deal of the technical analysis of  \cite{BeaGre} is absorbed by this remark.
\end{remark}

\section{Conormality and regularity}
\label{sec:conormality_and_regularity}

With Proposition \ref{prop:symbol_homogeneity} in hand, we  now prove the standard properties of pseudodifferential operators for our class $\Psi^m_H(M)$.  This begins with decay estimates on the Fourier transform $\hat{\PP}$.

\subsection{Symbol estimates}

Recall that a \emph{homogeneous norm} $\|\cdot\|_H$ on a graded vector space $V$ is a norm   that scales under the canonical dilations as
\[
  \| \delta_\lambda ( \xi ) \|_H = \lambda \|\xi\|_H,   \qquad \qquad \xi\in V,~ \lambda\in\RR_+.
\]
        
\begin{corollary}
\label{cor:decay_conditions}
  Let $\PP\in \PPsi_H^m(M)$ be homogeneous on the nose outside $[-1,1]$.  Fix a local exponential coordinate system $(x,\xi,t) \in M_0 \times V \times \RR \into \tHM \times \RR$ as in Section \ref{sec:charts} and denote the dual coordinates by $(x,\eta,t) \in M_0 \times V^* \times \RR \into \tsHM \times \RR$.  The Fourier transform $\hat{\PP}$ in these coordinates satisfies the following estimates: For any compact $K\subset M_0$ and multi-indices $\bfa,\bfb, \bfc \in \NN^n$, $k\in\NN$, there is $C = C_{K,\bfa,\bfb,\bfc,k} >0$ such that for all $(x,\eta,t) \in K \times V^* \times \RR$,
\begin{equation}
 \label{eq:decay_conditions}
  |\eta^\mathbf{a} \partial_\eta^\bfb \partial_x^\bfc \partial_{t}^k \hat{\PP}(x,\eta,t)| \leq C (1+\|\eta\|_H+|t|)^{m+|\bfa|_H-|\bfb|_H-k}, 
\end{equation}
  where $\|\slot\|_H$ is any homogeneous norm on $V^*$.
\end{corollary}

\begin{proof}
 From Proposition \ref{prop:symbol_homogeneity} the function $\eta^\mathbf{a} \partial_\eta^\bfb \partial_x^\bfc \partial_{t}^k \hat{\PP}$ is equal, modulo a function in $\sS_\rho(\tsHM\times\RR)$, to a function that is homogeneous at infinity of order $m+|\bfa|_H-|\bfb|_H-k$ with respect to $(\beta_\lambda)$.
\end{proof}

We are now very close to the classical theory of pseudodifferential operators of type $\Psi^m_{\rho,\delta}$.    To the expert, the results that follow will be no surprise.


\subsection{Conormality}
\label{sec:conormality}

The first consequence of the estimates \eqref{eq:decay_conditions} is that our $H$-pseudodifferential kernels are proper distributions, i.e.\ are both $r$- and $s$-fibred.   For this, we follow a strategy of Lescure-Manchon-Vassout \cite{LesManVas}, showing that they have wavefront sets that are conormal to the diagonal.  Here is the relevant lemma.  

\begin{lemma}
\label{decomposability}
Let $G$ be a Lie groupoid.  Suppose $v\in \sD'(G)$ has proper support, singular support contained in $G^{(0)}$, and wave front set contained in the conormal to $G^{(0)}$:
  \begin{equation}
  \label{eq:conormality}
  \WF(v) \subseteq \{(x,\eta) \in T^*G|_{G^{(0)}} \;\st (\eta,\xi) = 0 \text{ for all } \xi \in T_xG^{(0)} \}.
 \end{equation}
 Then $v$ is a proper distribution.
\end{lemma}

\begin{proof}
 Since $T(G^{(0)})$ is complementary to the tangent space of the $r$-fibre at every point, the conormal bundle of $G^{(0)}$ intersects trivially with the conormal bundle of any $r$-fibre.  Now apply \cite[Proposition 7]{LesManVas} to conclude that $v$ can be disintegrated as an $r$-fibred distribution.  Likewise, $v$ disintegrates as an $s$-fibred distribution.
 \end{proof}

We now prove that the conormality condition \eqref{eq:conormality} 
holds for essentially homogeneous $r$-fibred distributions on $\TTHM$.  
Fix a transverse measure $\mu$ on the base of $\TTHM$, so that we have the integration maps $\mu_r:\sE'_r(\TTHM) \to \sD'(\TTHM)$ and $\mu_s:\sE'_s(\TTHM) \to \sD'(\TTHM)$ as in Section \ref{sec:proper_distributions}.

\begin{proposition}
 \label{prop:conormality}
 Let $\PP\in\PPsi_H^m(M)$.  The wave front set of the distribution $\mu_r(\PP) \in \sD'(\TTHM)$ is conormal to the object space, i.e., 
 \begin{align*}
  \WF(\mu_r(\PP)) \subseteq \big\{ (x,\eta, t) \in &T_{(x,t)}^*(\TTHM) \st (x,t)\in M\times\RR = \TTHM^{(0)} \\
    &\text{ and } (\eta,\xi) = 0 \text{ for all }
    \xi \in T_{(x,t)}(\TTHM^{(0)}) \big\}.
 \end{align*}
 Thus $\PP$ is a proper $r$-fibred distribution.
\end{proposition}
\begin{proof}

We work in local exponential coordinates $M_0\times V \times \RR \to \TTHM$.  It suffices to prove the proposition for $\PP$ homogeneous on the nose outside $[-1,1]$, so we have the decay estimates \eqref{eq:decay_conditions}.

Let us determine the wave front set of $\mu_r(\PP)$ at a point $(x_0,t_0) \in M\times \RR$.   We may replace $\PP$ with $\chi\PP$, where $\chi\in C^\infty(M\times\RR)$ is some smooth bump function on the base that is $1$ near $(x_0,t_0)$ and supported in some compact set $K\times I \subset M_0\times \RR$. Since the fibrewise Fourier transform is $C^\infty(M\times\RR)$-linear, we still have decay conditions on $\widehat{\chi\PP}$, in particular,
\[
  |\partial_x^\mathbf{c} \partial_{t}^k(\widehat{\chi\PP})(x,\eta,t)| \leq C_{\bfc,k} (1+\|\eta\|_H)^{m}
\]
for every $\bfc\in\NN^n$, $k\in\NN$.  The factor of $|t|$ has been removed from the right hand side, since the support of $\chi\PP$ is bounded in $t$.
 
 Now we apply the Fourier transform $\sF_{(x,t)\to(y,s)}$ in the base variables.  We obtain bounds
 \[
  |y^\mathbf{c} s^k \sF_{(x,\xi,t) \to (y,\eta,s)} (\chi\widetilde{\PP})(y,\eta,s)| \leq C_{\bfc,k} \, \mu(K) \, (1+\|\eta\|_{V^*})^{m} 
 \]
 for every $\bfc\in\NN^n$, $k\in\NN$.  It follows that every $(y,\eta,s)\in\RR^n\times\RR^n\times\RR$ with $(y,s)\neq(0,0)$ admits a conical neighbourhood upon which $\sF_{(x,\xi,t) \to (y,\eta,s)} (\chi\PPtilde)$ has rapid decay.  This proves the conormality of the wave front set.
\end{proof}

\begin{corollary}
Elements of $\PPsi^m_H(M)$ are proper $r$-fibred distributions.  Thus, $H$-pseudodifferential kernels $P\in\Psi^m_H(M)$ are semiregular in both variables (see Proposition \ref{prop:regular}).
\end{corollary}

\subsection{Algebra structure}

Properness is needed to define the algebra structure on $\Psi^\bullet_H(M)$, since we need to know that $\Cp^\infty(\TTHM;\Omega_r)$ is a two-sided ideal.  This, plus the fact that $\alpha_{\lambda*}$ is an algebra automorphism of $\sE'_{r,s}(\TTHM)$ immediately yields the following.

\begin{theorem}
\label{thm:algebra}
For any $m,m'\in\RR$ we have $\PPsi^m_H(M) * \PPsi_H^{m'}(M) \subseteq \PPsi_H^{m+m'}(M)$.  Thus $\Psi^\bullet_H(M) = \bigcup_{m\in\ZZ}\Psi^m_H(M)$ is a filtered algebra.  
\end{theorem}

\begin{proposition}
 \label{prop:symbol_is_algebra_map}
 If $P \in \Psi_H^m(M)$ and $Q \in \Psi^{m'}_H(M)$ then $\sigma_{m+m'}(P*Q) = \sigma_m(P) * \sigma_{m'}(Q)$.
\end{proposition}

\begin{proof}
 If $\PP \in \PPsi^m_H$ and $\QQ \in \PPsi_H^{m'}$ extend $P$ and $Q$, respectively, then $\PP*\QQ \in \PPsi_H^{m+m'}$ extends $P*Q$.  The result follows.
\end{proof}

\subsection{Regularity}
\label{sec:regularity}

Another consequence of the decay conditions \eqref{eq:decay_conditions}
is the regularity of $H$-pseudodifferential kernels of large negative order.

\begin{lemma}
 \label{lem:regularity}
With $k \geq 0$, suppose $\PP \in \PPsi_H^{-d_H-k-1}(M)$ is supported in some local exponential coordinate patch, and let $\PPtilde$ be its pullback under local exponential coordinates $(x,\xi,t) \in M_0\times V \times \RR \supseteq \VV  \to \UU \subset \TTHM$ as in Definition \ref{def:local_pullback}.  For any $\mathbf{a}, \mathbf{b} \in \NN^n$, $c\in\NN$ with $|\mathbf{a}|_H \leq k$, we have
 \[
  \partial_\xi^\mathbf{a} \partial_x^\mathbf{b} \partial_{t}^c \PPtilde \in C(M_0\times V \times\RR ; \Omega_s).
 \]
\end{lemma}

\begin{proof}
We assume $\PP$ is homogeneous on the nose outside $[-1,1]$, so we have the decay estimates \eqref{eq:decay_conditions}.  For any multi-indices $\bfa,\bfb\in\NN^n$, $c\in\NN$ with $|\mathbf{a}|_H \leq k$, and for each $(x,t)\in {M_0}\times\RR$, the function $\eta \mapsto \eta^\mathbf{a} \partial_x^\mathbf{b} \partial_{t}^c \hat{\PP}(x,\eta,t)$ is bounded in absolute value by $C' ( 1 + \|\eta\|_H )^{-d_H-1}$, which is integrable (see \cite{NagSte}).  
Taking fibrewise Fourier transforms gives $\partial_\xi^\mathbf{a} \partial_x^\mathbf{b} \partial_{t}^c {\PPtilde} \in C({M_0}\times V \times \RR; \Omega_r)$.
\end{proof}

\begin{theorem}
 \label{thm:k-regularity}
Let $M$ be a smooth manifold equipped with a Lie filtration of depth $N$ and homogeneous dimension $d_H$.  
For any $k\geq0$ we have $\PPsi_H^{-d_H-kN-1}(M) \subseteq C^k(\TTHM;\Omega_r)$ and $\Psi_H^{-d_H-kN-1}(M) \subseteq C^k(M\times M;\Omega_r)$.
\end{theorem}

\begin{proof}
It suffices to consider $\PP\in\PPsi_H^{-d_h-kN-1}(M)$ supported in some local exponential coordinate patch,  thanks to Lemma \ref{lem:local_exponential_support}.  The result then follows from Lemma \ref{lem:regularity} since  $|\bfa|_H \leq N |\bfa|$.
\end{proof}

\begin{corollary}
\label{cor:smoothing}
If we put $\Psi^{-\infty}_H(M) = \bigcap_{m\in\ZZ} \Psi^m_H(M)$, then $\Psi_H^{-\infty}(M) = \Psi^{-\infty}(M) = \Cp^\infty(M\times M;\Omega_r)$ is the algebra of properly supported smoothing operators on $M$.
\end{corollary}


\section{$H$-ellipticity and Parametrices}
\label{sec:H-ellipticity}
\label{sec:parametrices}


The convolution algebra $\sE'_r(G)$ on any Lie groupoid $G$ has an identity element $I$, namely the evaluation at the units:
\[
  \ip{I,f}(x) := f(x), \qquad f\in C^\infty(G), ~x\in G^{(0)}.
\]
We will denote the identity elements of $\sE'_r(\TTHM)$, $\sE_r'(M\times M)$ and $\sE'_r(\THM)$ by $\II$, $\II_1=I$ and $\II_0$, respectively.

\begin{definition}
 \label{def:H-elliptic}
 An $H$-pseudodifferential operator $P\in\Psi^m_H(M)$  will be called {\em $H$-elliptic} if its principal cosymbol $\sigma_m(P)$ admits a convolution inverse in $\sE'_r(T_HM)/\Cp^\infty(T_HM;\Omega_r)$.
\end{definition}

\begin{lemma}
 \label{lem:inverse_is_homogeneous}
 The convolution inverse of a principal cosymbol of order $m$ (if it exists) is a principal cosymbol of order $-m$.  
\end{lemma}

\begin{proof}
 Let $u \in \sE'_r(\THM)$ represent a principal cosymbol of order $m$ with convolution inverse $v \in \sE'_r(\THM)$ modulo $\Cp^\infty(\THM;\Omega_r)$.  
Working modulo $\Cp^\infty(T_HM;\Omega_r)$, we have
 \begin{multline*}
  \delta_{\lambda*}v \equiv (\delta_{\lambda*}v) * u * v \equiv (\delta_{\lambda*}v) * (\lambda^{-m}\delta_{\lambda*}u) * v \\
  \equiv \lambda^{-m} \delta_{\lambda*}(I) * v = \lambda^{-m}v,
 \end{multline*}
 for all $\lambda\in\RR^\times_+$.  The result follows.
\end{proof}

\begin{lemma}
 \label{lem:parametrix-1}
 Let $P\in\Psi^m_H(M)$ be $H$-elliptic.  Then there is $Q\in\Psi^{-m}_H(M)$ such that $P * Q - I$ and $Q*P - I$ are in $\Psi^{-1}_H(M)$.
\end{lemma}

\begin{proof}
 Fix $\PP \in \PPsi_H^m$ extending $P$.  From the previous lemma and the surjectivity of the cosymbol map (Lemma \ref{lem:cosymbols_extend}) we can find $\QQ\in \PPsi_H^{-m}$ such that $\PP_0 * \QQ_0 \equiv \II_0 \equiv \QQ_0 * \PP_0$ modulo $\Cp^\infty(T_HM;\Omega_r)$.  Put  $Q = \QQ_1$.  The cosymbols $\sigma_0(P*Q - I)$ and $\sigma_0(Q*P - I)$ vanish, and the result follows from Lemma \ref{cor:lower_order}.
\end{proof}

\begin{definition}
 \label{def:parametrix}
 Let $P\in\Psi^m_H(M)$.  An $H$-pseudodifferential operator $Q\in\Psi^{-m}_H(M)$ is called a {\em parametrix for $P$} if $P*Q -I \in \Psi^{-\infty}(M)$ and $Q*P -I \in \Psi^{-\infty}(M)$.
\end{definition}

To construct parametrices, we need convergence of asymptotic expansions.

\begin{definition}
\label{def:asymptotic}
An \emph{asymptotic series} of $H$-pseudodifferential operators is a series of the form $P_0 + P_1 + P_2 + \cdots$ where $P_k\in\Psi_H^{m_k}(M)$ and the orders $m_k$ decrease to $-\infty$.  We will say \emph{$P\in\Psi^\bullet_H(M)$ is an asymptotic limit} of this series, denoted
\[
  P \sim P_0 + P_1 + P_2 + \cdots,
\]
if $P - (\sum_{k=1}^{j-1} P_k) \in \Psi_H^{m_j}(M)$ for all $j\in\NN$.
\end{definition}

\begin{theorem}
 \label{thm:completeness}
Every asymptotic series admits an asymptotic limit.
\end{theorem}

\begin{proof}
The proof follows a standard strategy, so we shall be brief.
It suffices to consider asymptotic series with $P_k \in \Psi^{-d_H-kN-1}_H(M)$.
Extend each $P_k$ to $\QQ_k\in \PPsi^{-d_H-kN-1}(\TTHM;\Omega_r)$,
and then let  $\PP_k = t^{kN}\QQ_k \in \PPsi^{-d_H-1}_H(M)$.
We will also assume that $\PP_k$ vanishes to order $k$ on $\TTHM^{(0)}$,
which can always be achieved after adding an element of $C_p^\infty(\TTHM;\Omega_r)$, thanks to Theorem \ref{thm:k-regularity}.

Using cut off functions that are equal to $1$ in a neighbourhood of $\TTHM^{(0)}$,
but have shrinking supports as $k$ increases, we can modify the sequence $\PP_k$ such that
$\sum_{k=0}^\infty \PP_k$ converges to some $\PP\in \PPsi^{-d_H-1}_H(M)$.
When checking convergence, it is useful to know that we may concentrate on values $t\in [-1,1]$ because of  Proposition \ref{prop:homogeneous_on_the_nose}.
\end{proof}

\begin{theorem}
 \label{thm:parametrices}
 Every $H$-elliptic pseudodifferential operator $P\in\Psi^m_H(M)$ admits a parametrix.
\end{theorem}

\begin{proof}
Pick $Q \in\Psi^{-m}_H(M)$ as in Lemma \ref{lem:parametrix-1}.
Put $R:=I-P*Q$ and let $R^k$ denote the $k$th convolution power of $R$.
By Theorem \ref{thm:completeness} the series $\sum_{k=0}^\infty R^{k}$ admits an asymptotic limit $A \in \Psi^0_H(M)$.
Then $Q' = Q*A$ is a parametrix for $P$.
\end{proof}

\begin{corollary}
 \label{cor:hypoelliptic}
 Every $H$-elliptic differential operator is hypoelliptic. \qed
\end{corollary}


\section{Differential Operators}
\label{sec:diff_ops}

In this section we spell out the details for  differential operators.  To avoid technicalities, we shall assume that $M$ is compact%
\footnote{The specific technicality we are avoiding is that the order of a differential operator on a noncompact manifold may increase to infinity as we stray outward.}.

The symbolic calculus of differential operators on filtered manifolds
is most naturally understood with reference to the {\em Lie algebroid} of $\TTHM$.
As is evident in \cite{VanYun:groupoid}, it is easier to construct this Lie algebroid
than it is to construct the groupoid.
In particular,  local exponential coordinates are not necessary
to understand the calculus of differential operators.

The reader unfamiliar with Lie algebroids might consult \cite{Mackenzie} or \cite{MoeMrc}.  
Briefly, if $G$ is a Lie groupoid, then the associated Lie algebroid is the vector bundle $\AG = \mathrm{ker}\,ds|_{G^{(0)}}$ restricted to the space of units $G^{(0)}$,
equipped with a {\em bracket} and an {\em anchor}.
Sections $\Gamma^\infty(\AG)$ are in one-to-one correspondence with left-invariant vector fields on $G$ that are in the kernel of $ds$.
The bracket of left-invariant vector fields on $G$ induces a bracket operation on $\Gamma^\infty(\AG)$.
We shall not make use of the anchor.


\subsection{The universal enveloping algebra of a Lie algebroid}
\label{sec:enveloping_algebra}

Let $G$ be a Lie group.  It is well known that there are algebra isomorphisms between (\emph{a}) the set of left-invariant differential operators on $G$, (\emph{b}) the universal enveloping algebra $\sU(\lie{g})$, and (\emph{c}) the set of distributions on $G$ supported at the identity (an algebra under convolution).  All of these have analogues for Lie groupoids.  

The analogue of (\emph{a}) is the algebra $\DO_s(G)^G$ of right-invariant differential operators on $G$ that are tangent to the $s$-fibres%
\footnote{
The algebra $\DO_s(G)^G$ is denoted by $\mathrm{Diff}(G)$ in \cite{NisWeiXu}.
}.
For (\emph{b}), one takes the universal enveloping algebra $\sU(\AG)$ of the Lie algebroid $\AG$.  For a definition, see e.g.
\cite[Section 3]{NisWeiXu}.
For (\emph{c}) we make the following definition.

\begin{definition}
We write $\sE'_r(G)^{(0)}$ for the set of $r$-fibred distributions on $G$ that are supported on the units space $G^{(0)}$.
\end{definition}

In the following theorem, most of which appears in \cite{NisWeiXu}, we use $T_sG := \ker ds \subseteq TG$ to denote the bundle of vectors tangent to the $s$-fibres of $G$.  We will also need the transpose maps $t:\sE'_s(G) \leftrightarrows \sE'_r(G)$ defined by 
\[
 P \mapsto P^t := P \circ \iota^*
\]
where $\iota:G \to G$ is the groupoid inverse.  Note that transpose is an anti-isomorphism of the convolution algebras.

\begin{theorem}
\label{thm:enveloping_algebra}
 There are canonical algebra isomorphisms
 \begin{equation}
 \label{eq:UAG-isomorphisms}
  \sU(\AG) \cong \DO_s(G)^G \cong \sE'_r(G)^{(0)}.
 \end{equation}
 Specifically, the left-hand isomorphism is the unique extension of the map $\Gamma^\infty(\AG) \to \Gamma^\infty(T_sM)$ induced by right-translation:
 \[
  X \mapsto \tilde{X}, \qquad \text{where } \tilde{X}_\gamma := R_{\gamma*}X_{r(\gamma)}.
 \]
 The right-hand isomorphism is given by
 \begin{equation}
 \label{eq:IoD}
  D \mapsto \left[ (I\circ D)^t = I \circ D\circ \iota^*  : C^\infty(G) \to C^\infty(G^{(0)}) \right]
 \end{equation}
 where $I: C^\infty(G) \to C^\infty(G^{(0)})$ is evaluation on the units.  The inverse map is $u\mapsto D_u$ where
 \begin{equation}
  \label{eq:Du}
  D_u\varphi(\gamma) := \ang{ u^t_{r(\gamma)} , \varphi(\slot\gamma) }.
 \end{equation}
\end{theorem}


%
%
%

\begin{proof}
For the left-hand isomorphism in Equation \eqref{eq:UAG-isomorphisms}, see \cite{NisWeiXu}, Section 3.  For the right-hand isomorphism, note that standard results about distributions supported at a point imply that a smooth family $u=(u_x)_{x\in G^{(0)}} \in \sE'_r(G)^{(0)}$ is necessarily given by evaluation at the units of some smooth differential operator $P\in\DO_s(G)$.  From here, direct calculations show that the maps \eqref{eq:IoD} and \eqref{eq:Du} are mutually inverse algebra morphisms.  We leave the details to the reader.
\end{proof}


\subsection{Principal symbols}
\label{sec:diff_op_symbols}

In the case of the pair groupoid, there is also an algebra isomorphism
\[
 \DO_s(M\times M)^{M\times M} \cong  \DO(M), 
\]
where $\DO(M)$ designates the algebra of differential operators on $M$.  Specifically, a differential operator $D\in\DO(M)$ is mapped to the differential operator $\tilde{D}\in\DO_r(M\times M)$ that acts as $D$ on the first factor.  
The resulting isomorphism $\DO(M) \cong \sE'_r(M\times M)^{(0)}$ sends a differential operator to its Schwartz kernel. 

If $M$ is a filtered manifold, then we have a filtration on the sections of the Lie algebroid $\Gamma^\infty(TM)$ that is compatible with the Lie bracket.  This extends to a filtration on the universal enveloping algebra $\sU(TM)\cong \DO(M)$.  We refer to this as the \emph{$H$-order} of a differential operator.  In the case of a trivially filtered manifold, this is the usual order.  We denote the space of differential operators of $H$-order $\leq m$ by $\DO_H^m(M)$.

Since the associated graded Lie algebra of $\Gamma^\infty(TM)$ is $\Gamma^\infty(\tHM)$, the grading maps $ \sigma_m:\Gamma^\infty(H^m) \to \Gamma^\infty(\tHM)$
extend to a family of algebra grading maps 
\[
 \sigma_m : \DO_H^m(M) \to \sU(\tHM).
\]
Concretely, if $D = X_1\ldots X_k$ is a monomial of vector fields where $X_i\in\Gamma^\infty(H^{m_i})$, then  $\sigma_m(D) = \sigma_{m_1}(X_1)\ldots\sigma_{m_k}(X_k)$.  This defines the {\em principal part} of a differential operator of $H$-order $m$.


\subsection{Differential operators as $H$-pseudodifferential operators}
\label{sec:diff_ops_homogeneous}

Let $D$ be a differential operator on $M$ of $H$-order $m$ with Schwartz kernel $P \in \sE'_r(M\times M)^{(0)}$.  Its principal part $\sigma_m(D) \in \sU(\tHM)$ can be identified, via Theorem \ref{thm:enveloping_algebra}, with an element of $\sE'_r(\THM)^{(0)}$, which we will denote by $\sigma_m(P)$ since we will shortly see it is precisely the principal cosymbol of $P$.

We  associate to $D$ a family $(\PP_t)_{t\in\RR}$ of distributions $\PP_t \in \sE'_r(\TTHM|_t)$ as follows:
\begin{equation}
\label{eq:ssigma}
  \PP_t = \begin{cases}
   t^m P , & t\neq 0 \\
   \sigma_m(P), & t=0.
  \end{cases}
\end{equation}

\begin{proposition}
Let $P\in\sE'_r(M\times M)$ be the Schwartz kernel of a differential operator $D$ of $H$-order $m$.  Then the family $\PP$ above is smooth, \emph{i.e.}, it defines an element of $\sE'_r(\TTHM)^{(0)}$.  Moreover, it is homogeneous of order $m$ on the nose.  In particular, $P\in\Psi^m_H(M)$.
\end{proposition}

\begin{proof}
 If $D$ is order $0$, this is proven in Example \ref{ex:identity}.  
 
 If $D$ is a vector field, then smoothness of $\PP$ follows from Example 10 of \cite{VanYun:groupoid}.  This family is obviously homogeneous on the nose for $t\neq0$, and continuity ensures it is also homogeneous at $t=0$.  
 
 Since functions and vector fields are generating for $\sU(TM)\cong \DO(M)$, the result follows.
\end{proof}

We can now give the proof of Proposition \ref{prop:DOs} from the Introduction, which we generalize here to the filtered case.

\begin{proposition}
 \label{prop:H-DOs}
A semiregular kernel $P \in \sE_r'(M\times M)$ is the Schwartz kernel of a differential operator of $H$-order $\leq m$ if and only if $P = \PP|_{t=1}$ for some $\PP \in \sE_r'(\TTM)$ that is homogeneous on the nose of weight $m$.
\end{proposition}

\begin{proof}
It only remains to prove the reverse implication.  If $\PP\in \sE_s'(\TTHM)$ is homogeneous on the nose of weight $m$, then $\supp(\PP)$ is invariant under the zoom action. Since the only proper orbits (in the sense of Definition \ref{def:proper_subset}) of the zoom action are those contained in the unit space, $\supp(\PP) \subseteq \TTHM^{(0)}$. Thus $P = \PP_1$ is the kernel of a differential operator $D$ on $M$. 

Suppose $D$ has $H$-order $\leq n$. Then the family $t^n P$ on $\TTHM|_{\RR^\times}$ extends smoothly to $\sigma_n(P)$ at $t=0$.  But $t^nP = t^{n-m}\PP$ for $t\neq0$, so by continuity, $\sigma_n(P) = 0$.  Thus $D$ has $H$-order $\leq n-1$.  Repeating this, $D$ has $H$-order $\leq m$.
\end{proof}


\section{Comparison with the classical calculus}
\label{sec:classical_PsiDOs}

In this section, we prove that our $H$-pseudodifferential operators coincide with the classical pseudodifferential operators in the case of a trivially filtered manifold $M$ (Theorem \ref{thm:PsiDOs} of the Introduction).  

For the terminology of classical pseudodifferential operators, we follow \cite[\S18.1]{Hormander:III}. 
By restricting to local exponential coordinates, we may assume that $M$ is an open domain in $\RR^n$.  We equip $TM = M \times \RR^n$ with the standard connection $\nabla$. 

To begin with, suppose $P$ is (the Schwartz kernel of) a classical pseudodifferential operator of order $m$ on $M$.    
As remarked in Example \ref{ex:full_symbol}, its symbol (in the usual sense, \cite[\S18.1]{Hormander:III}) is given by
\[
  a= \hat{P} \in C^\infty(M \times \hat{\RR}^n).
\]
By the definition of {\em classical} pseudodifferential operator, $a$ admits an asymptotic symbol expansion
\begin{equation}
 \label{eq:polyhomogeneous}
a(x,\eta)\sim\sum_j a_j(x,\eta),
\end{equation}
where $a_j$ is homogeneous of order $m-j$ in $\eta$ outside of $\|\eta\|\leq1$ (see Definition 18.1.5 of \cite{Hormander:III}).  

The inverse Fourier transforms $\check{a}_j = \sF^{-1}_r a_j \in \sS'_r(TM)$ satisfy
\[
  \delta_{\lambda*} \check{a}_j - \lambda^{m-j} \check{a}_j \in \sS_r(TM;\Omega_r), 
  \qquad \text{for all } \lambda\in\RR^\times
\]
and have singular support on the zero section $M\times\{0\}$ (see e.g.\ \cite{Taylor:microlocal}).   Let us define the constant family of distributions
\[
 \PPtilde_j(x,\xi,t) = \varphi_1(x,\xi) \check{a}_j(x,\xi),
\]
where $\varphi_1 \in \Cp(TM)$ is some bump function supported in a domain of injectivity for $\Exp^\nabla$.  Since $\PPtilde_j$ is essentially homogeneous with respect to the zoom action $\tilde\zoom$ on $TM\times\RR$, its push-forward $\PP_j$ via exponential coordinates belongs to $\PPsi^{m-j}_H(M)$.  The series $\sum_j \PP_j|_{t=1}$ admits an asymptotic limit $Q\in\Psi^m_H(M)$ in the sense of Definition \ref{def:asymptotic}.  

By the regularity of $H$-pseudodifferential kernels (Theorem \ref{thm:k-regularity}) and the analogous result for classical pseudodifferential operators, we see that for $k\gg 0$ both $P$ and $Q$ differ from $\Exp^\nabla_* (\sum_{j=1}^k\check{a}_j)$ by a kernel in $C^{k-m-n}(M\times M;\Omega_r)$.  Thus $P-Q  \in \Psi^{\infty}(M)$.  This shows that $P$ is an $H$-pseudodifferential operator.

Conversely, suppose $P\in \Psi^m_H(M)$ is an $H$-pseudodifferential operator.  By Corollary \ref{cor:decay_conditions},
$a:=\hat{P}$ is in the symbol class $S^m(M\times \hat{\RR}^n)$.  It remains to prove that $a$ admits a polyhomogeneous expansion $a\sim\sum_ja_j$.

Choose $\PP\in\PPsi^m_H(M)$ with $P=\PP_1$.  Put $B_0 = \hat{\PP}$.  Note that $B_0|_{t=0}$ is homogeneous of weight $m$ modulo $\sS_{r'}(T^*M)$ for the dilations $\delta'_\lambda$ on $T^*M$.  It is therefore equal, modulo $\sS_{r'}(T^*M)$, to a function $a_0\in C^\infty(T^*M)$ that is exactly homogeneous of weight $m$ outside of $\|\eta\| \leq 1$ (see \cite[Proposition 2.2]{Taylor:microlocal}).  If we define $A_0\in C^\infty(T^*M\times \RR)$ by $A_0|_t  = a_0$ for all $t$, then
\[
\beta_\lambda^* A_0 - \lambda^{m}A_0 \in \sS_{r'}(t^*M\times\RR) \qquad \text{for all } \lambda\in\RR^\times_+.
\]

It follows that $B_1 := t^{-1}(B_0-A_0)$ extends smoothly at $t=0$ to a function on $T^*M\times\RR$ satisfying
\[
  \beta_\lambda^* B_1 - \lambda^{m-1}B_1 \in \sS_{r'}(t^*M\times\RR) \qquad \text{for all } \lambda\in\RR^\times_+.
\]
Note that $B_1|_{t=1} = a-a_0$.  Reasoning as for Corollary \ref{cor:decay_conditions}, we get $a-a_0 \in S^{m-1}(M\times\hat\RR^n)$.  Also, $B_1|_{t=0}$ is equal, modulo $\sS_{r'}(T^*M)$, to a function $a_1\in C^\infty(T^*M)$ that is exactly homogeneous of weight $m-1$ outside of $\|\eta\| \leq 1$.

Repeating, we define $A_1\in C^\infty(T^*M\times\RR)$ by $A_1|t = a_1$ for all $t$ and then $B_2 := t^{-1}(B_1-A_1)$.  Restricting to $t=1$ gives $a-a_0-a_1 \in S^{m-2}(M\times\hat\RR^n)$.  Continuing in this fashion, we obtain an asymptotic symbol expansion $a \sim \sum_{j=0}^m a_j$. 

We have now proven the following.

\begin{theorem}
 Let $M$ be a smooth manifold without boundary, equipped with the trivial filtration: $H^1 = TM$.  The algebra $\Psi^\bullet_H(M)$ of $H$-pseudo\-differential operators is precisely the algebra of (Schwartz kernels of) classical pseudodifferential operators on $M$.
\end{theorem}

Likewise, our construction yields the Heisenberg calculus of Beals and Greiner \cite{BeaGre} in the case where the filtration of $M$ is given by a codimension one subbundle $H$ of $TM$:
\[
 H^0= 0 \quad \leq \quad H^1 = H \quad \leq \quad H^2=TM.
\]
This is because operators in their calculus have symbols which are polyhomogeneous for the appropriate dilations on the duals of the osculating groups.  We omit the details.


\section{Pseudodifferential calculi on Lie groupoids}
\label{sec:filtered_groupoid}

In singular geometric situations, such as foliated manifolds, manifolds with boundary or corners, and stratified manifolds, the appropriate algebras of pseudodifferential operators are associated to Lie groupoids other than the pair groupoid.  

A full history of this enormous body of work would take too much space, so with apologies to the many authors involved, here is an extremely rapid summary.
Various classes of pseudodifferential operators associated to singular spaces have been introduced by analysts, beginning with Melrose's seminal work on the $b$-calculus \cite{Melrose:transformation, Melrose:APS, MelPia}.  Inspired by Connes' approach to index theory for foliations \cite{Connes:integration}, these constructions were subsequently reinterpreted in terms of Lie groupoids and put into a general framework \cite{NisWeiXu, Monthubert:groupoids, AmmLauNis, DebLesRoc}.  
Thanks to this reinterpretation, our construction is readily adapted to produce such pseudodifferential calculi.  One can even decorate the resulting calculi with a filtration on the Lie algebroid to obtain pseudodifferential operators of filtered type on singular spaces.  

We caution, however, that this direction is somewhat separate from our main goal, which is to simplify the machinery of the Heisenberg calculus and its filtered analogues.  Note also that the machinery of the $b$-calculus, for instance, is based not just on the construction of pseudodifferential operators, but also on the boundary conditions needed to obtain Fredholm operators in these singular situations.  The methods described in this paper have nothing to say about these boundary conditions.  

Nevertheless, since the generalization is easy and may be useful in future work, we will take the time to explain the modifications necessary for our construction to yield right-invariant $H$-pseudodifferential operators on Lie groupoids.  Proofs will be omitted, since they are entirely analogous to the manifold case.  Also, to conserve space, we will take up the notation and terminology of Section 9 of \cite{VanYun:groupoid}.  This section is therefore not independent of that paper.  
We will state the results for Hausdorff Lie groupoids, although they apply equally well for almost differentiable \cite{NisWeiXu} (or longitudinally smooth \cite{Monthubert:groupoids}) groupoids.

\begin{definition}
 \cite{VanYun:groupoid}
 A Lie groupoid $G\rightrightarrows M$ is called \emph{filtered} if its Lie algebroid $\AG$ is equipped with a filtration by subbundles
 \[
  0 = \mathrm{A}^0 G \leq \mathrm{A}^1 G \leq \cdots \leq \mathrm{A}^N G = \AG
 \]
 such that the module of sections $\Gamma^\infty(\mathrm{A}^\bullet G)$ is a filtered Lie algebra.
\end{definition}

The associated graded bundle of $\AG$ is denoted $\aHG$ and called the \emph{osculating Lie algebroid}.  It is again a smooth bundle of nilpotent Lie algebras, and exponentiates to a smooth bundle of connected, simply connected nilpotent Lie groups called the \emph{osculating groupoid} and denoted $\AHG$.  
Being a graded bundle, $\aHG$ admits a one parameter family of dilations $(\delta_\lambda)_{\lambda\in\RR}$, and these induce a one-parameter family of groupoid endomorphisms of $\AHG$,.

The analogue of the $H$-tangent groupoid in this context is the \emph{$H$-adiabatic groupoid},
\[
  \AAHG := \AHG \times \{0\} \;\sqcup\; G\times\RR^\times,
\]
which is a Lie groupoid with smooth structure as defined in \cite{VanYun:groupoid}.   It admits a \emph{zoom action} $(\alpha_\lambda)_{\lambda\in\RR^\times_+}$
\begin{align}
 \label{eq:adiabatic_zoom}
  \zoom_\lambda (g,t) &= (g,\lambda^{-1}t) && \text{if } g\in G,~t\neq0,\\
  \zoom_\lambda (x,\xi,0) &= (x, \delta_\lambda (\xi) , 0) && \text{if } x\in M,~\xi\in\AHG_x,~t=0. \nonumber
\end{align}
which is a smooth one parameter family of Lie groupoid automorphisms.

\begin{definition}
Fix $m\in\RR$.  We define
\begin{align*}
 \PPsi^m_H(G) := 
  \{ \PP \in \sE'_r(\AAHG) \st \alpha_{\lambda*} \PP - \lambda^m \PP \in \Cp^\infty(&\AAHG;\Omega_r)\\
   & \text{for all }\lambda\in\RR^\times_+ \}.
\end{align*}
An $r$-fibred distribution $P \in \sE'_r(G)$ is called an \emph{$H$-pseudodifferential kernel of order $\leq m$ on $G$} if $P = \PP_1$ for some $\PP\in\PPsi^m_h(G)$.  The set of $H$-pseudodifferential kernels of order $\leq m$ will be denoted $\Psi^m_H(G)$.
\end{definition}

An $H$-pseudodifferential kernel $P\in\Psi_H^m(G)$ acts on $\Cp^\infty(G)$ by left convolution: for $\varphi \in \Cp^\infty(G)$,
\[
  (\Op(P) \varphi) (\gamma) :=  \ang{ u^t_{r(\gamma)} , \varphi(\slot\gamma) };
\]
compare Equation \eqref{eq:Du}.  Then $\Op(P)$ is a right-invariant operator on the $s$-fibres of $G$, called the \emph{$H$-pseudodifferential operator} associated to $P$.  Again, we will often blur the distinction between pseudodifferential operators and their kernels.

\begin{definition}
The space of \emph{principal cosymbols of order $m$ on $G$} is defined as
\begin{multline*}
 \Sigma_H^m(G) := \{ K \in \sE'_r(\AHG) / \Cp^\infty(\AHG;\Omega_r) \st \\
   \delta_{\lambda*} K = \lambda^m K \text{ for all } \lambda\in\RR^\times_+ \}.
\end{multline*}
\end{definition}

We have the following analogue of Corollary \ref{cor:principal_symbol}, Lemma \ref{lem:cosymbols_extend} and Corollary \ref{cor:lower_order}.

\begin{proposition}
  Let $P\in\Psi_H^m(G)$.  Define the \emph{principal cosymbol of $P$} by $\sigma_m(P) := [\PP_0] \in \Sigma_H^m(G)$, where $\PP \in \PPsi^m_H(G)$ with $\PP_1=P$.  Then $\sigma_m(P)$ is  well-defined independent of the choice of $\PP$.  The resulting \emph{cosymbol map} $\sigma_m$ fits into a short exact sequence
  \[
   0 \longrightarrow    \Psi^{m-1}_H(G) \longrightarrow \Psi^m_H(G) \stackrel{\sigma_m}{\longrightarrow} \Sigma^m_H(G) \longrightarrow 0.
  \]
\end{proposition}

Using the groupoid exponential of $\AAHG$, Equation (15) of \cite{VanYun:groupoid} provides a diffeomorphism of an open set in $\aaHG \times \RR$ into $\AAHG$ which linearizes the $r$-fibres locally near the unit space.  If we fix a local trivialization $M_0 \times V$ of the graded bundle $\aaHG$, then we obtain local exponential coordinates $(x,\xi,t) \in M_0 \times V \times \RR \supset \VV \into \AAHG$ as in Section \ref{sec:charts}.

Taking the Fourier transform of $\PP\in\Psi^m_H(G)$ in these local exponential coordinates, we obtain exactly the same decay estimates on $\hat{\PP}$ as in Corollary \ref{cor:decay_conditions}.  This immediately implies that the wavefront set of $\PP$ is conormal to the unit space---so that $\PP$ is a proper $r$-fibred distribution---and that the singularities at the units satisfy the same regularity conditions as in Lemma \ref{lem:regularity}.

Therefore, convolution induces a well-defined product $\PPsi^m_H(G) * \PPsi^{m'}_H(G) \to \PPsi^{m+m'}_H(G)$.  This restricts to give products on both $\Psi_H^\bullet(G)$ and the cosymbol algebra $\Sigma_H^\bullet(G)$.  Also, $\Psi^{-\infty}_H(G) := \bigcap_{m\in\ZZ}\Psi^m_H(G) = \Cp^\infty(G;\Omega_r)$ is the smooth convolution algebra of $G$, and every asymptotic series in $\Psi^\bullet_H(G)$ admits an asymptotic limit (\emph{cf}. Theorem \ref{thm:completeness}).  As a  result we have the following theorem.

\begin{theorem}
\label{thm:groupoid_hypoellipticity}
 Let $P\in\Psi^m_H(G)$ be an $H$-pseudodifferential operator on the Lie groupoid $G$.  If its principal symbol $\sigma_m(P)$ admits a convolution inverse in $\Sigma_H^\bullet(G)$, then $P$ admits a parametrix $Q\in \Psi^{-m}_H(G)$ such that $PQ-I$ and $QP-I$ belong to $\Psi^{-\infty}(G) = \Cp^\infty(G;\Omega_r)$.
\end{theorem}


In the case where the groupoid $G$ is trivially filtered, \emph{i.e.},
\[
  H^0 = 0 \quad \leq \quad  H^1 = \AG,
\]
our pseudodifferential algebra $\Psi_H^m(G)$ coincides with the algebra $\Psi^\infty_\mathrm{prop}(G)$ of properly supported pseudodifferential operators on $G$ from \cite{NisWeiXu}.  This follows by combining the argument of Section \ref{sec:classical_PsiDOs} with the characterization of pseudodifferential kernels in Theorem 7 of \cite{NisWeiXu}.  The latter result only applies to compactly supported distributions on $G$ (which correspond to \emph{uniformly supported pseudodifferential operators} in the terminology of \cite{NisWeiXu}), but one can pass to the properly supported case by using a partition of unity on $G^{(0)}$.

\begin{example}
The above construction can be applied to the $b$-groupoid of a manifold $M$ with boundary $\partial M$. 
 The \emph{$b$-groupoid} ${}^bG \rightrightarrows M$ of Monthubert \cite{Monthubert:groupoids} 
is a Lie groupoid whose Lie algebroid is the $b$-tangent groupoid ${}^bTM$ of Melrose \cite{Melrose:APS}.  The simplest way to define it is via its module of sections:
\[
  \Gamma^\infty({}^bTM)  =\{ X \in \Gamma^\infty(TM) \st X|_{\partial M} \text{ is tangent to the boundary} \}.
 \]
If we apply our machinery to  the adiabatic groupoid
 \[
  \AA{}^bG = ( {}^bG \times \RR^\times ) \;\sqcup\; ({}^bTM \times \{0\}).
 \]
it yields a pseudodifferential calculus $\Psi_H^\bullet({}^bG)$.  
This calculus is exactly the small $b$-calculus $\Psi_b^\bullet(M)$  of Melrose \cite{Melrose:APS}

\end{example}

\begin{remark}
 It is now completely natural to imagine pseudodifferential calculi on filtered manifolds with boundary, or with more general geometries at infinity (\emph{cf.}, \cite{AmmLauNis}, \cite{DebLesRoc}).  We will leave that direction for a future paper.
\end{remark}



\bibliographystyle{alpha}
\bibliography{refs}

\end{document}